\documentclass[12pt]{amsart}
\usepackage{enumerate,amsmath,xpatch,relsize}

\xapptocmd\normalsize{%
	\abovedisplayskip=12pt plus 3pt minus 9pt
	\abovedisplayshortskip=0pt plus 3pt
	\belowdisplayskip=12pt plus 3pt minus 9pt
	\belowdisplayshortskip=7pt plus 3pt minus 4pt
}{}{}
\theoremstyle{definition}
\newtheorem{definition}{Definition}[section]

\theoremstyle{plain}
\newtheorem{theorem}[definition]{Theorem}
\newtheorem{corollary}[definition]{Corollary}
\newtheorem{lemma}[definition]{Lemma}

\numberwithin{equation}{section}

\title {Radii of Starlikeness and Convexity of   Bessel Functions }
\author[V. Madaan]{Vibha Madaan}
\address{Department of Mathematics, University of Delhi, Delhi--110 007, India}
\email{vibhamadaan47@gmail.com}

\author[A. Kumar]{Ajay Kumar}
\address{Department of Mathematics, University of Delhi, Delhi--110 007, India}
\email{akumar@maths.du.ac.in}

\author[V. Ravichandran]{V. Ravichandran}
\address{Department of Mathematics, National Institute of Technology, Tiruchirappalli--620015, India}
\email{vravi68@gmail.com, ravic@nitt.edu}

\keywords{$q$-Bessel function, Lommel function, Lemniscate of Bernoulli, Janowski function, Radius of starlikeness}
\subjclass[2010]{30C10; 30C15; 30C45}

\thanks{The first author is supported by University Grants Commission(UGC), UGC-Ref. No.:1069/(CSIR-UGC NET DEC, 2016).}

\allowdisplaybreaks
\begin{document}
\maketitle
\begin{abstract}
The radii of starlikeness and convexity associated with lemniscate of Bernoulli and the Janowski function, $(1+Az)/(1+Bz)$ for $-1\leq B<A\leq 1$, have been determined for normalizations of $q$-Bessel function, Bessel function of first kind of order $\nu$, Lommel function of first kind and Legendre polynomial of odd degree.
\end{abstract}
	
\section{Introduction}
Let $\mathbb{D}$ be the unit disk  in $\mathbb{C}$  and $\mathcal{A}$ be the class of all analytic functions $f\colon\mathbb{D}\to\mathbb{C}$ normalized  by $f(0)=0$ and $f'(0)=1$. The class $\mathcal{S}$ is the subclass of $\mathcal{A}$ consisting of univalent functions. A function $f\in \mathcal{A}$ is \emph{starlike} if $f(\mathbb{D})$ is starlike with respect to the origin. Various subclasses of starlike function can be unified by making use of the concept of subordination. A function $f\in\mathcal{A}$ is said to be \emph{subordinate} to a function $g\in\mathcal{A}$, written as $f(z)\prec g(z)$, if there is a Schwarz function $w$ with $w(0)=0$ such that $f(z)=g(w(z))$. If $g$ is a univalent function, then $f(z)\prec g(z)$ if and only if $f(0)=g(0)$ and $f(\mathbb{D})\subset g(\mathbb{D})$.  For an analytic function $\varphi$, let $\mathcal{S}^*(\varphi)$ be the class of all analytic functions satisfying  $zf'(z)/f(z)\prec\varphi(z)$.  The class $\mathcal{K}(\varphi)$ is the class of all analytic functions satisfying $1+zf''(z)/f'(z)\prec\varphi(z)$. These classes include respectively several well-known subclasses of starlike and convex  functions. For example, the class $\mathcal{S}^*_\mathcal{L}:=\mathcal{S}^*(\sqrt{1+z})$  denotes the class of  lemniscate starlike functions introduced and studied  by Sok\'{o}l and Stankiewicz \cite{MR1473947} and the class $\mathcal{K}(\sqrt{1+z})$,  denoted by $\mathcal{K}_\mathcal{L}$, is the class of  lemniscate convex functions. For $-1\leq B<A\leq 1$, the class $\mathcal{S}^*[A,B]:=\mathcal{S}^*((1+Az)/(1+Bz))$ is the class of Janowski starlike functions and  the class $\mathcal{K}[A,B]:=\mathcal{K}((1+Az)/(1+Bz))$ is the class of Janowski convex functions (see  \cite{MR0267103}).

Given a class of functions $\mathcal{M}\subset \mathcal{A}$ and a function $f\in\mathcal{A}$, the $\mathcal{M}$-radius of the function $f$ is the largest number $r$ with  $0\leq r \leq 1$ such that $f_r\in\mathcal{M}$, where $f_r(z):=f(rz)/r$.  For $\mathcal{M}=\mathcal{S}^*_\mathcal{L}$, the $\mathcal{M}$-radius of $f$,  denoted by $r^*_\mathcal{L}(f)$, is called  the radius of \emph{lemniscate starlikeness}. It is indeed the largest $r$ with $0\leq r\leq 1$ such that \[ \left|\left(\frac{zf'(z)}{f(z)}\right)^2-1\right| <1\quad \quad (|z|<r). \] For $\mathcal{M}=\mathcal{K}_\mathcal{L}$, the $\mathcal{M}$-radius of $f$,  denoted by $r^c_\mathcal{L}(f)$, is called  the radius of \emph{lemniscate convexity}. It is indeed the largest $r$ with $0\leq r\leq 1$ such that \[ \left|\left(1+\frac{zf''(z)}{f'(z)}\right)^2-1\right|<1 \quad \quad (|z|<r) . \] When $\mathcal{M}=\mathcal{S}^*[A,B]$ or $\mathcal{K}[A,B]$, the respective radii, denoted by $r^*_{A,B}(f)$ and  $r^c_{A,B}(f)$, are called the radii of  \emph{Janowski starlikeness} and \emph{Janowski convexity} of the function $f$. These are respectively the largest $r$ with $0\leq r\leq 1$ satisfying \[  \left|\frac{(zf'(z)/f(z))-1}{A-Bzf'(z)/f(z)}\right|<1 \quad \text{and}\quad  \left|\frac{zf''(z)/f'(z)}{A-B(1+zf''(z)/f'(z))}\right|<1 \quad \quad (|z|<r). \] For more details on the radius problems, one may refer to \cite{MR2879136,MR0704183,MR3620283}.

We intend to look certain normalized Bessel functions for which our aim is to find the radii of lemniscate starlikeness, lemniscate convexity, Janowski starlikeness and Janowski convexity. Consider the Bessel function of first kind of order $\nu$ which is a particular solution of the second order homogeneous Bessel differential equation $z^2w''(z)+zw'(z)+(z^2-\nu^2)w(z)=0$, where $\nu$ is an unrestricted (real or complex) number. The function has an infinite series expansion given by \[ J_\nu(z)=\sum\limits_{n\geq0}\frac{(-1)^n}{n!\Gamma(n+\nu+1)}\left(\frac{z}{2}\right)^{2n+\nu},  \] where $z\in\mathbb{C}$ and $\nu\in\mathbb{C}$ such that $\nu\neq 0,-1,-2,\ldots$. Watson's treatise \cite{MR1349110} made a comprehensive study of Bessel function of first kind. Geometric properties of Bessel function of first kind such as univalence, convexity, starlikeness etc.\@ have been studied in \cite{MR2656410}. Since the Bessel function $J_\nu$ is not normalized, we will consider the following normalizations:
\begin{align}
f_\nu(z)&=(2^\nu\Gamma(\nu+1)J_\nu(z))^{1/\nu}=z-\frac{1}{4\nu(\nu+1)}z^3+\cdots\text{,}\quad{} \nu \neq 0\label{Bessel norm f}\\
g_\nu(z)&=2^\nu\Gamma(\nu+1)z^{1-\nu}J_\nu(z)=z-\frac{1}{4(\nu+1)}z^3+\cdots\label{Bessel norm g}
\intertext{and}
h_\nu(z)&=2^\nu\Gamma(\nu+1)z^{1-\nu/2}J_\nu(\sqrt{z})=z-\frac{1}{4(\nu+1)}z^2+\cdots\text{.}\label{Bessel norm h}
\end{align}
Clearly, the functions $f_\nu$, $g_\nu$ and $h_\nu$ belong to the class $\mathcal{A}$. It can also be  noted that $f_\nu(z)=\exp(\frac{1}{\nu}\log(2^\nu\Gamma(\nu+1)J_\nu(z))$, where $\log$ represents the principle branch of logarithm function. Throughout every multivalued function is taken with the principle branch. Baricz \emph{et al.\@} \cite {MR3182021}  determined the radii of starlikeness of $f_\nu$, $g_\nu$ and $h_\nu$ for $\nu>-1$. Also, Baricz and Sz\'{a}sz \cite{MR3252850} studied radius of convexity of Bessel function of first kind. Bohra \emph{et al.\@} \cite{MR3764714} have obtained the $\mathcal{ST}(k,\alpha)$ radius, that is, the radius of $k$-starlikeness of order $\alpha$  and the $k$-uniformly convex radius, $\mathcal{UCV}(k,\alpha)$ radius, of the normalized Bessel functions of first kind.

Now consider the Jackson and Hahn-Exton $q$-Bessel functions given by
\begin{align*}
J_\nu^{(2)}(z;q)&=\frac{(q^{\nu+1};q)_\infty}{(q;q)_\infty}\sum_{n\geq1}\frac{(-1)^n (\frac{z}{2})^{2n+\nu}}{(q^{\nu+1};q)_n(q;q)_n}q^{n(n+\nu)},\\
J_\nu^{(3)}(z;q)&=\frac{(q^{\nu+1};q)_\infty}{(q;q)_\infty}\sum_{n\geq1}\frac{(-1)^n z^{2n+\nu}}{(q^{\nu+1};q)_n(q;q)_n}q^{n(n+1)/2},
\end{align*}
where $z\in\mathbb{C}$, $v>-1$, $q\in(0,1)$ and $(a;q)_0=1$, $(a;q)_n=\prod_{k=1}^n(1-aq^{k-1})$  and $(a;q)_\infty=\prod_{k\geq1}(1-aq^{k-1})$. The $q$-Bessel functions are analytic $q$-extensions of the classical Bessel function of first kind $J_\nu$. With $\nu>-1$, for a fixed $z$, we have $J_\nu^{(2)}((1-z)q;q)\rightarrow J_\nu(z)$ and $J_\nu^{(3)}((1-z)q;q)\rightarrow J_\nu(2z)$ as $q\nearrow1$. The following normalizations of the $q$-Bessel functions $J_\nu^{(s)}(z)$ are considered for $s=2,3$:
\begin{align}
f_\nu^{(2)}(z;q)&=(2^\nu c_\nu(q)J_\nu^{(2)}(z;q))^{1/\nu}\label{Bessel norm f2},\quad{} \nu\neq 0\\
g_\nu^{(2)}(z;q)&=2^\nu c_\nu(q)z^{1-\nu}J_\nu^{(2)}(z;q)\label{Bessel norm g2},\\
h_\nu^{(2)}(z;q)&=2^\nu c_\nu(q)z^{1-\nu/2}J_\nu^{(2)}(\sqrt{z};q)\label{Bessel norm h2},
\end{align}
where $c_\nu(q)=(q;q)_\infty/(q^{\nu+1};q)_\infty$. Similarly, we have
\begin{align}
f_\nu^{(3)}(z;q)&=(c_\nu(q)J_\nu^{(3)}(z;q))^{1/\nu}\label{Bessel norm f3},\quad{}\nu\neq 0\\
g_\nu^{(3)}(z;q)&=c_\nu(q)z^{1-\nu}J_\nu^{(3)}(z;q)\label{Bessel norm g3},\\
h_\nu^{(3)}(z;q)&=c_\nu(q)z^{1-\nu/2}J_\nu^{(3)}(\sqrt{z};q).\label{Bessel norm h3}
\end{align}
Clearly the functions $f_\nu^{(s)}(\cdot;q)$, $g_\nu^{(s)}(\cdot;q)$ and $h_\nu^{(s)}(\cdot;q)$ belong to the class $\mathcal{A}$ for $s\in\{2,3\}$.  Baricz \emph{et al.\@} \cite{MR3423439} derived the radii of starlikeness and convexity of these $q$-Bessel functions and also observed that for the radius problem, the geometric properties of Jackson $q$-Bessel function and Hahn-Exton $q$-Bessel function are similar. The bound on the radii of starlikeness and convexity of some $q$-Bessel function have been obtained in \cite{MR3684469} and \cite{IO}, respectively.

The Lommel function of first kind $s_{\mu,\nu}(z)$, given by \[ s_{\mu,\nu}(z)=\frac{z^{\mu+1}}{(\mu-\nu+1)(\mu+\nu+1)} {}_1F_2\left(1;\frac{\mu-\nu+3}{2},\frac{\mu+\nu+3}{2};\frac{-z^2}{4}\right),\quad{}z\in\mathbb{C}, \] with ${}_1F_2$ being the hypergeometric function, is a particular solution of the inhomogeneous Bessel differential equation $z^2w''(z)+zw'(z)+(z^2-\nu^2)w(z)=z^{\mu+1}$, where $\mu\pm\nu$ is not a negative odd integer. The geometric properties of Lommel function $s_{\mu-\frac{1}{2},\frac{1}{2}}$ have been observed in \cite{MR3503704}. Baricz \emph{et al.\@} \cite{MR3596935} studied the properties of zeros of Lommel function $s_{\mu-\frac{1}{2},\frac{1}{2}}$ and Struve function $\textbf{H}_\nu$ and their derivatives and obtained the radii of convexity of certain normalizations of these functions. For $\nu=1/2$, following analytic normalizations of the Lommel function of first kind have been considered to determine the radii of starlikeness and convexity of the respective functions associated with lemniscate of Bernoulli and the Janowski function $(1+Az)/(1+Bz)$: 	
\begin{align}
f_\mu(z)&=f_{\mu-\frac{1}{2},\frac{1}{2}}(z)=\left(\mu(\mu+1)s_{\mu-\frac{1}{2},\frac{1}{2}}(z)\right)^{\frac{1}{\mu+\frac{1}{2}}}, \quad (\mu\neq -1/2)                                   \label{Lommel norm f}\\
g_\mu(z)&=g_{\mu-\frac{1}{2},\frac{1}{2}}(z)=\mu(\mu+1)z^{-\mu+\frac{1}{2}}s_{\mu-\frac{1}{2},\frac{1}{2}}(z)                                                        \label{Lommel norm g}
\intertext{and}
h_\mu(z)&=h_{\mu-\frac{1}{2},\frac{1}{2}}(z)=\mu(\mu+1)z^{\frac{3-2\mu}{4}}s_{\mu-\frac{1}{2},\frac{1}{2}}(\sqrt{z}).                                                   \label{Lommel norm h}
\end{align}
In \cite{MKR2}, the conditions on the respective parameters are derived so that generalized normalized Bessel function and generalized normalized Lommel function are lemniscate convex and lemniscate Carath\'{e}odory in $\mathbb{D}$.

The Legendre polynomials are the solutions of the Legendre differential equation  $(1-z^2)w''(z)-2zw'(z)+n(n+1)w(z)=0$, where $n$ is a non-negative integer. The Legendre polynomial $P_n$ is symmetric or antisymmetric, that is $P_n$ satisfies $P_n(-z)=(-1)^n P_n(z)$. Let the Legendre polynomial $P_{2n-1}$ is normalized by \[ \mathcal{P}_{2n-1}(z)=\frac{P_{2n-1}(z)}{P'_{2n-1}(0)}=z+a_2z^2+\cdots+a_{2n-1}z^{2n-1}. \]
In \cite{MR3905364}, the authors have determined the radius of starlikeness, convexity and uniform convexity of the Legendre polynomial of odd degree.

Since the $q$-Bessel function, Bessel function of first kind, Lommel function of first kind and the Legendre polynomial are all entire functions, hence whenever the $\mathcal{M}$-radius of the respective normalized function exceeds $1$, then function is said to possess the property of class $\mathcal{M}$ in disk $\mathbb{D}_r:=\{z: |z|<r\}$.

In the following section, using some Mittag-Leffler expansions for quotients of the normalized Bessel functions, the radii of lemniscate starlikeness and convexity have been determined for the normalized $q$-Bessel functions, normalized Bessel functions of first kind, the normalized Lommel functions and normalized Legendre polynomial of odd degree. In section \ref{janowski}, the Janowski starlike and Janowski convex radii have been obtained for these normalized functions. The fact that zeros of the normalized Bessel functions are real and interlaced with the zeros of their derivatives is used as a key tool in the proofs. The same holds for the zeros of normalized Lommel function of first kind and its derivative.

\section{Lemniscate Starlikeness and Lemniscate Convexity}\label{lemniscate}
This section deals with the problem to find the radii of lemniscate starlikeness and lemniscate convexity of the normalizations of the $q$-Bessel function, Bessel function of first kind, Lommel function of first kind and the Legendre polynomial of odd degree. According to \cite[Lemma 1]{MR3423439}, for $\nu>-1$, the Hadamard factorizations of the functions $J^{(2)}_\nu(z;q)$ and $J^{(3)}_\nu(z;q)$ for $z\in\mathbb{C}$ are given by
\begin{equation}\label{jqv}
J_\nu^{(2)}(z;q)=\frac{z^\nu}{2^\nu c_\nu(q)} \prod_{n\geq1} \left(1-\frac{z^2}{\xi^2_{\nu,n}(q)}\right)  \text{ and }  J_\nu^{(3)}(z;q) =\frac{z^\nu}{c_\nu(q)} \prod_{n\geq1} \left(1-\frac{z^2}{\zeta^2_{\nu,n}(q)}\right),
\end{equation}
where $\xi_{\nu,n}(q)$ and $\zeta_{\nu,n}(q)$ are $n^{\text{th}}$ positive zeros of the function $J_\nu^{(2)}(\cdot;q)$ and $J_\nu^{(3)}(\cdot;q)$, respectively satisfying the inequality $\xi_{\nu,1}(q)<\xi_{\nu,2}(q)<\cdots$ and $\zeta_{\nu,1}(q)<\zeta_{\nu,2}(q)<\cdots$.

The following result gives the lemniscate starlike radii of the functions $f_\nu^{(s)}$, $g_\nu^{(s)}$ and $h_\nu^{(s)}$  for $s\in\{2,3\}$. 	
\begin{theorem}\label{lem starlike q Bessel}
	For $0<q<1$, let $\xi_{\nu,1}(q)$ and $\zeta_{\nu,1}(q)$ denote the first positive zero of the $q$-Bessel functions $J^{(2)}_\nu(\cdot;q)$ and $J^{(3)}_\nu(\cdot;q)$, respectively. For $s\in\{2,3\}$, if $\nu>0$, the function $f_\nu^{(s)}(\cdot;q)$ and for $\nu>-1$, the functions $g_\nu^{(s)}(\cdot;q)$ and $h^{(s)}_\nu(\cdot;q)$ have their lemniscate starlike radii; $r^*_\mathcal{L}(f^{(s)}_\nu)$, $r^*_\mathcal{L}(g^{(s)}_\nu)$ and $r^*_\mathcal{L}(h^{(s)}_\nu)$, respectively as the smallest positive root of the equations (respectively)
	\begin{align}
	&r^2\left(\frac{d}{dr}J_\nu^{(s)}(r;q)\right)^2-4r\nu J_\nu^{(s)}(r;q) \frac{d}{dr}J_\nu^{(s)}(r;q) +2\nu^2(J_\nu^{(s)}(r;q))^2=0,\label{lem star Bessel eqn fs}\\
	\begin{split}
	&\left(r\frac{d}{dr}J_\nu^{(s)}(r;q)-\nu J_\nu^{(s)}(r;q)\right)^2\\
	&\quad{}\quad{}-2\left(r \frac{d}{dr}J_\nu^{(s)}(r;q)-\nu J_\nu^{(s)}(r;q)\right)J_\nu^{(s)}(r;q)-(J_\nu^{(s)}(r;q))^2=0\label{lem star Bessel eqn gs}
	\end{split}
	\end{align}
	and
	\begin{equation}\label{lem star Bessel eqn hs}
	\begin{split}
	&\left(\sqrt{r}\frac{d}{dr}J_\nu^{(s)}(\sqrt{r};q)-\nu J_\nu^{(s)}(\sqrt{r};q)\right)^2\\
	&\quad{}\quad{}-4\left(\sqrt{r} \frac{d}{dr}J_\nu^{(s)}(\sqrt{r};q)-\nu J_\nu^{(s)}(\sqrt{r};q)\right) J_\nu^{(s)}(\sqrt{r};q)-4(J_\nu^{(s)}(\sqrt{r};q))^2=0.
	\end{split}
	\end{equation}
	For $s=2,3$, $r^*_\mathcal{L}(f_\nu^{(s)})$, $r^*_\mathcal{L}(g_\nu^{(s)})$ and $r^*_\mathcal{L}(h_\nu^{(s)})$ are the unique roots of the equations \eqref{lem star Bessel eqn fs}, \eqref{lem star Bessel eqn gs} and \eqref{lem star Bessel eqn hs}, respectively in $(0,t)$, where $t=\xi_{\nu,1}(q)$ for $s=2$ and $t=\zeta_{\nu,1}(q)$ for $s=3$.
\end{theorem}
\begin{proof}
	The proofs for the cases $s=2$ and $s=3$ are almost the same except for the difference of the zeros of the corresponding functions. The proof presented is for $s=2$. The case $s=3$ follows similarly. For simplicity, the notations $J_\nu(z;q):=J_\nu^{(2)}(z;q)$, $f_\nu(z;q):=f_\nu^{(2)}(z;q)$, $g_\nu(z;q):=g_\nu^{(2)}(z;q)$ and $h_\nu(z;q):=h_\nu^{(2)}(z;q)$ are used and the derivatives are with respect to the first component. The factorization \eqref{jqv} of the function $J_\nu(\cdot;q)$ implies
	\begin{equation}\label{Bessel expr jq}	\frac{zJ'_\nu(z;q)}{J_\nu(z;q)}=\nu-\sum_{n\geq1}\frac{2z^2}{\xi_{\nu,n}^2(q)-z^2}.
	\end{equation}
	Keeping in view the normalizations \eqref{Bessel norm f2}, \eqref{Bessel norm g2}, \eqref{Bessel norm h2}, it follows from the equation \eqref{Bessel expr jq} that
	\begin{align}
	\frac{zf'_\nu(z;q)}{f_\nu(z;q)}&=\frac{1}{\nu}\frac{zJ'_\nu(z;q)}{J_\nu(z;q)}=1-\frac{1}{\nu}\sum_{n\geq1}\frac{2z^2}{\xi_{\nu,n}^2(q)-z^2},\label{Bessel expr f2}\\
	\frac{zg'_\nu(z;q)}{g_\nu(z;q)}&=1-\nu+\frac{zJ'_\nu(z;q)}{J_\nu(z;q)}=1-\sum_{n\geq1}\frac{2z^2}{\xi_{\nu,n}^2(q)-z^2},\label{Bessel expr g2}\\
	\frac{zh'_\nu(z;q)}{h_\nu(z;q)}&=1-\frac{\nu}{2}+\frac{1}{2}\frac{\sqrt{z}J'_\nu(\sqrt{z};q)}{J_\nu(\sqrt{z};q)}=1-\sum_{n\geq1}\frac{z}{\xi_{\nu,n}^2(q)-z}.\label{Bessel expr h2}
	\end{align}
	By virtue of \eqref{Bessel expr f2}, for $|z|<\xi_{\nu,1}(q)$, by simple computations, the inequality
	\begin{equation*}\label{Bessel expr mod f}
	\begin{split}
	\left|\left(\frac{zf'_\nu(z;q)}{f_\nu(z;q)}\right)^2-1\right|&\leq \frac{1}{\nu^2}\left(\sum_{n\geq1}\frac{2|z|^2}{\xi_{\nu,n}^2(q)-|z|^2}\right)\left(\sum_{n\geq1}\frac{2|z|^2}{\xi_{\nu,n}^2(q)-|z|^2}+2\nu\right)\\
	&=\left(\frac{|z|f'_\nu(|z|;q)}{f_\nu(|z|;q)}\right)^2-4\left(\frac{|z|f'_\nu(|z|;q)}{f_\nu(|z|;q)}\right)+3
	\end{split}
	\end{equation*} holds for $\nu>0$ and $n\in\mathbb{N}$, where the last equality holds by replacing $z$ by $|z|$ in equation \eqref{Bessel expr f2}. Similar inequalities hold for the functions $g_\nu(\cdot;q)$ and $h_\nu(\cdot;q)$ for $\nu>-1$.
	
	If the function $p(\cdot;q)$ collectively represents the functions $f_\nu(\cdot;q)$, $g_\nu(\cdot;q)$ and $h_\nu(\cdot;q)$, then by the above relation, with $r^*$ being the smallest positive root of the equation \[	\left(\frac{rp'(r;q)}{p(r;q)}\right)^2-4\left(\frac{rp'(r;q)}{p(r;q)}\right)+2=0,\] the inequality $|(rp'(r;q)/p(r;q))^2-1|<1$ is true for $|z|<r^*$. From equations \eqref{Bessel expr f2}, \eqref{Bessel expr g2} and \eqref{Bessel expr h2}, note that the zeros of the above equation for the functions $f_\nu(\cdot;q)$, $g_\nu(\cdot;q)$ and $h_\nu(\cdot;q)$ coincide with those of equations \eqref{lem star Bessel eqn fs}, \eqref{lem star Bessel eqn gs} and \eqref{lem star Bessel eqn hs}, respectively with $s=2$. This implies that the lemniscate starlike radii of these functions, $r^*_\mathcal{L}(f_\nu)$, $r^*_\mathcal{L}(g_\nu)$ and $r^*_\mathcal{L}(h_\nu)$, are smallest positive roots of the mentioned equations with $s=2$.
	
	In particular for the function $f_\nu(\cdot;q)$, the function $u_\nu(\cdot;q) \colon (0,\xi_{\nu,1}(q)) \to \mathbb{R}$, defined by $ u_\nu(r;q):=(rf'_\nu(r;q)/f_\nu(r;q))^2-4 rf'_\nu(r;q)/f_\nu(r;q) +2$, is continuous and is strictly increasing function of $r$ for $\nu>0$, as \[ u'_\nu(r;q) =\frac{2}{\nu} \sum_{n\geq1} \frac{4r\xi_{\nu,n}^2(q)}{(\xi^2_{\nu,n}(q)-r^2)^2} \left(1+\frac{1}{\nu}\sum_{n\geq1}\frac{2r^2}{\xi^2_{\nu,n}(q)-r^2}\right)>0.\] Also, $\lim_{r\searrow0}u_\nu(r;q)=-1<0$ and $\lim_{r\nearrow \xi_{\nu,1}(q)}u_\nu(r;q)=\infty>0$. Thus, the Intermediate Value Theorem ensures the existence of the unique zero of $u_\nu(\cdot;q)$ in $(0,\xi_{\nu,1}(q))$ and hence, the lemniscate starlike radius of the function $f_\nu(z;q)$; $r^*_\mathcal{L}(f_\nu)$, is the unique zero of $u_\nu(r;q)$ in $(0,\xi_{\nu,1}(q))$ or of the equation \eqref{lem star Bessel eqn fs} for $s=2$. In the similar manner, the existence of root of the other equations can be verified and this is how we can determine radii of lemniscate starlikeness for the functions $f_\nu^{(s)}$, $g_\nu^{(s)}$ and $h_\nu^{(s)}$ for $s=2,3$. 	
\end{proof}

%>>>>>>>>>>>>>>>>>>>>>>>>>>>>>>>>>>>>>>>>>>>>>>>>>>>>>>>>>>>>>>>>>>>>>>>>>>>>>>>>>>>>>>>>>>>>>>>>>>
For $s=2,3$ and $\nu>0$, the Hadamard's factorizations of $dJ^{(s)}_\nu(z;q)/dz$, \cite[Lemma 7]{MR3423439}, are given by
\begin{equation}\label{Bessel djq}
\frac{d}{dz}J^{(2)}_\nu(z;q)=\frac{\nu z^{\nu-1}}{2^\nu c_\nu(q)}\prod_{n\geq1}\left(1-\frac{z^2}{\xi'^2_{\nu,n}(q)}\right) \quad \text{and} \quad  \frac{d}{dz}J^{(3)}_\nu(z;q)=\frac{\nu z^{\nu-1}}{c_\nu(q)}\prod_{n\geq1}\left(1-\frac{z^2}{\zeta'^2_{\nu,n}(q)}\right),
\end{equation}
where $\xi'_{\nu,n}(q)$ and $\zeta'_{\nu,n}(q)$ are the $n^\text{th}$ positive zeros of the functions $dJ^{(2)}_\nu(z;q)/dz$ and $dJ^{(3)}_\nu(z;q)/dz$, respectively. Using \cite[Theorem 4.3]{MR0649849} and \cite[Lemma 9]{MR3423439}, the zeros of $dJ^{(2)}_\nu(z;q)/dz$ and $J^{(2)}_\nu(z;q)$ are interlaced and the relation $0<\xi'_{\nu,1}(q)<\xi_{\nu,1}(q)<\xi'_{\nu,2}(q)<\xi_{\nu,2}(q)<\cdots$ holds. By \cite[Theorem 3.7]{MR1293849}, a similar relation holds between the zeros of $dJ_\nu^{(3)}(z;q)/dz$ and $J^{(3)}_\nu(z;q)$.

The next result presents the radius for lemniscate convexity of the normalized $q$-Bessel functions $f^{(s)}_\nu$, $g^{(s)}_\nu$ and $h^{(s)}_\nu$ for $s=2,3$.
\begin{theorem}\label{lem convex q Bessel}
	For the $q$-Bessel function $J^{(2)}_\nu(\cdot;q)$ and $J^{(3)}_\nu(\cdot;q)$ with $0<q<1$, let $\xi'_{\nu,1}(q)$ and $\zeta'_{\nu,1}(q)$ denote the first positive zeros of $dJ^{(s)}_\nu(z;q)/dz$ for $s=2,3$, respectively and $\alpha_{\nu,1}(q)$ and $\gamma_{\nu,1}(q)$ be the first positive zeros of $z\cdot dJ^{(s)}_\nu(z;q)/dz+(1-\nu)J^{(s)}_\nu(z;q)$ for $s=2$ and $s=3$, respectively. Similarly, let $\beta_{\nu,1}(q)$ and $\delta_{\nu,1}(q)$ denote the first positive zeros of $z\cdot dJ^{(s)}_\nu(z;q)/dz+(2-\nu)J^{(s)}_\nu(z;q)$ with $s=2$ and $s=3$, respectively. For $s\in\{2,3\}$, the lemniscate convex radii for the function $f^{(s)}_\nu(\cdot;q)$ with $\nu>0$ and for the functions $g^{(s)}_\nu(\cdot;q)$ and $h^{(s)}_\nu(\cdot;q)$ with $\nu>-1$, $r^c_\mathcal{L}(f^{(s)}_\nu)$, $r^c_\mathcal{L}(g^{(s)}_\nu)$ and $r^c_\mathcal{L}(h^{(s)}_\nu)$ are respectively the smallest positive root of the equations
	\begin{equation}\label{lem convex Bessel eqn fs}
	\begin{split}
	&\left(\dfrac{rd^2J^{(s)}_\nu(r;q)/dr^2}{dJ^{(s)}_\nu(r;q)/dr}+\left(\dfrac{1}{\nu}-1\right)\dfrac{r dJ^{(s)}_\nu(r;q)/dr}{J^{(s)}_\nu(r;q)}\right)^2\\
	&\quad{}\quad{}\quad{}-2\left(\dfrac{rd^2J^{(s)}_\nu(r;q)/dr^2}{dJ^{(s)}_\nu(r;q)/dr}+\left(\dfrac{1}{\nu}-1\right)\dfrac{r dJ^{(s)}_\nu(r;q)/dr}{J^{(s)}_\nu(r;q)}\right)-1=0,
	\end{split}
	\end{equation}
	\begin{equation}\label{lem convex Bessel eqn gs}
	\begin{split}
	&\left(\frac{r^2d^2J^{(s)}_\nu(r;q)/dr^2+2(1-\nu)rdJ^{(s)}_\nu(r;q)/dr+\nu(\nu-1)J^{(s)}_\nu(r;q)}{r dJ^{(s)}_\nu(r;q)/dr+(1-\nu)J^{(s)}_\nu(r;q)}\right)^2\\
	&\quad{}\quad{}-2\left(\frac{r^2d^2J^{(s)}_\nu(r;q)/dr^2+2(1-\nu)rdJ^{(s)}_\nu(r;q)/dr+\nu(\nu-1)J^{(s)}_\nu(r;q)}{r dJ^{(s)}_\nu(r;q)/dr+(1-\nu)J^{(s)}_\nu(r;q)}\right)-1=0
	\end{split}
	\end{equation}
	and
	\begin{equation}\label{lem convex Bessel eqn hs}
	\begin{split}
	&\left(\frac{rd^2J^{(s)}_\nu(\sqrt{r};q)/dr^2+(3-2\nu)\sqrt{r}dJ^{(s)}_\nu(\sqrt{r};q)/dr+\nu(\nu-2)J^{(s)}_\nu(\sqrt{r};q)}{2(\sqrt{r}dJ^{(s)}_\nu(\sqrt{r};q)/dr+(2-\nu)J^{(s)}_\nu(\sqrt{r};q))}\right)^2\\
	&\quad{}-2\left(\frac{rd^2J^{(s)}_\nu(\sqrt{r};q)/dr^2+(3-2\nu)\sqrt{r}dJ^{(s)}_\nu(\sqrt{r};q)/dr+\nu(\nu-2)J^{(s)}_\nu(\sqrt{r};q)}{2(\sqrt{r}dJ^{(s)}_\nu(\sqrt{r};q)/dr+(2-\nu)J^{(s)}_\nu(\sqrt{r};q))}\right)-1=0.
	\end{split}
	\end{equation}
	In addition to that, with $s\in\{2,3\}$,  the radii $r^c_\mathcal{L}(f^{(s)}_\nu)$, $r^c_\mathcal{L}(g^{(s)}_\nu)$ and $r^c_\mathcal{L}(h^{(s)}_\nu)$ are unique positive roots of the equations \eqref{lem convex Bessel eqn fs}, \eqref{lem convex Bessel eqn gs} and \eqref{lem convex Bessel eqn hs}, respectively in $(0,t_1)$, $(0,t_2)$ and $(0,t_3)$ such that $t_1=\xi'_{\nu,1}(q)$, $t_2=\alpha_{\nu,1}(q)$ and $t_3=\beta_{\nu,1}(q)$ for $s=2$ and $t_1=\zeta'_{\nu,1}(q)$, $t_2=\gamma_{\nu,1}(q)$ and $t_3=\delta_{\nu,1}(q)$ for $s=3$.
\end{theorem}
\begin{proof}
	The proof for the case $s=2$ is given and the other one follows in the similar manner. Using equation \eqref{Bessel djq}, we obtain
	\begin{equation}\label{Bessel expr djq}
	1+\frac{zJ''_\nu(z;q)}{J'_\nu(z;q)}=\nu-\sum_{n\geq 1}\frac{2z^2}{\xi'^2_{\nu,n}(q)-z^2}.
	\end{equation}
	By means of normalization \eqref{Bessel norm f2} and equations \eqref{Bessel expr jq} and \eqref{Bessel expr djq}, it follows that
	\begin{equation}\label{Bessel expr df2}
	\begin{split}
	&1+\frac{zf''_\nu(z;q)}{f'_\nu(z;q)}=1+\frac{zJ''_\nu(z;q)}{J'_\nu(z;q)}+\left(\frac{1}{\nu}-1\right)\frac{zJ'_\nu(z;q)}{J_\nu(z;q)}\\
	&\quad{}=1-\sum_{n\geq1}\frac{2z^2}{\xi'^2_{\nu,n}(q)-z^2}-\left(\frac{1}{\nu}-1\right)\sum_{n\geq1}\frac{2z^2}{\xi^2_{\nu,n}(q)-z^2}.
	\end{split}
	\end{equation}
	Now, suppose that $\nu\in(0,1]$, then $\lambda=1/\nu-1\geq 0$. Using equation \eqref{Bessel expr df2} and triangle's inequality for $|z|<\xi'_{\nu,1}(q)<\xi_{\nu,1}(q)$, we get
	\begin{align*}
	\left|\left(1+\frac{zf''_\nu(z;q)}{f'_\nu(z;q)}\right)^2-1\right|&\leq \left(\sum_{n\geq1}\frac{2|z|^2}{\xi'^2_{\nu,n}(q)-|z|^2}+\lambda\sum_{n\geq1}\frac{2|z|^2}{\xi^2_{\nu,n}(q)-|z|^2}\right)^2\\
	&\quad{}+2\left(\sum_{n\geq1}\frac{2|z|^2}{\xi'^2_{\nu,n}(q)-|z|^2}+\lambda\sum_{n\geq1}\frac{2|z|^2}{\xi^2_{\nu,n}(q)-|z|^2}\right)
	\end{align*}
	holds. From equation \eqref{Bessel expr df2}, if $z$ is replaced by $|z|$, then the above inequality yields
	\begin{equation}\label{lem convex Bessel eqn f2}
	\left|\left(1+\frac{zf''_\nu(z;q)}{f'_\nu(z;q)}\right)^2-1\right| \leq \left(\frac{|z|f''_\nu(|z|;q)}{f'_\nu(|z|;q)}\right)^2-2\left(\frac{|z|f''_\nu(|z|;q)}{f'_\nu(|z|;q)}\right).
	\end{equation}
	Using the relation
	\begin{equation}\label{relation} \left|\frac{z}{a-z}-\lambda\frac{z}{b-z}\right|\leq\frac{|z|}{a-|z|}-\lambda\frac{|z|}{b-|z|} \end{equation}
	for $|z|\leq r<a<b$ and $0\leq \lambda<1$, from \cite[Lemma 2.1]{MR3641791}, we notice that the inequality \eqref{lem convex Bessel eqn f2} holds for the case when $\nu>1$ as well. Thus, for $\nu>0$ and $|z|<\xi'_{\nu,1}(q)$, the relation \eqref{lem convex Bessel eqn f2} holds. Therefore, the function $f_\nu(z;q)$ is lemniscate convex for $|z|<r_1$, where $r_1$ is the smallest positive root of
	\begin{equation}\label{lem convex Bessel eqn df2} 	
	\left(\frac{rf''_\nu(r;q)}{f'_\nu(r;q)}\right)^2-2\left(\frac{rf''_\nu(r;q)}{f'_\nu(r;q)}\right)-1=0.
	\end{equation}
	
	Further, we observe that the function $u_\nu(\cdot;q)\colon (0,\xi'_{\nu,1}(q))\to\mathbb{R}$, defined by $u_\nu(r;q):=(rf''_\nu(r;q)/f'_\nu(r;q))^2-2(rf''_\nu(r;q)/f'_\nu(r;q))-1$, is continuous and $\lim_{r\searrow0}u_\nu(r;q)=-1<0$ and $\lim_{r\nearrow \xi'_{\nu,1}(q)}u_\nu(r;q)=\infty>0$. By the  Intermediate Value Theorem, it follows that there exists a root of equation \eqref{lem convex Bessel eqn df2} in $(0,\xi'_{\nu,1}(q))$. Also, the function $u_\nu(r;q)$ is strictly increasing for $\nu>0$ because for $\nu>0$ and $r<\sqrt{\xi'_{\nu,1}(q)\xi_{\nu,1}(q)}$, the relation $\xi^2_{\nu,n}(q)(\xi'^2_{\nu,n}(q)-r^2)^2<\xi'^2_{\nu,n}(q)(\xi^2_{\nu,n}(q)-r^2)^2$ holds and from \eqref{Bessel expr df2} \begin{align*} u'_\nu(r;q)&>\left(8r\sum_{n\geq1}\left(\frac{\xi'^2_{\nu,n}(q)}{(\xi'^2_{\nu,n}(q)-r^2)^2}-\frac{\xi^2_{\nu,n}(q)}{(\xi^2_{\nu,n}(q)-r^2)^2}\right)\right)\\ &\quad{}\quad{}\quad{}\quad{}\left(2r^2\sum_{n\geq1}\left(\frac{1}{\xi'^2_{\nu,n}(q)-r^2}-\frac{1}{\xi^2_{\nu,n}(q)-r^2}\right)\right)>0.  \end{align*} Thus, the root is unique in $(0,\xi'_{\nu,1}(q))$ and hence the lemniscate convex radius of $f_\nu(z;q)$ is the unique root of \eqref{lem convex Bessel eqn df2} in $(0,\xi'_{\nu,1}(q))$. An equivalent form of this equation is given by equation \eqref{lem convex Bessel eqn fs} for $s=2$.
	
	We now determine the radii of lemniscate convexity of normalized $q$-Bessel functions $g_\nu(\cdot;q)$ and $h_\nu(\cdot;q)$. Since $g_\nu(z;q)$ is the normalization of the $q$-Bessel function $J_\nu(z;q)$ given by \eqref{Bessel norm g2}, therefore $g'_\nu(z;q)=2^\nu c_\nu(q)z^{-\nu}(zJ'_\nu(z;q)+(1-\nu)J_\nu(z;q))$. 	Using \cite[Lemma 8]{MR3423439}, the Hadamard factorization of $g'_\nu(z;q)$ is given by
	\begin{equation*}
	g'_\nu(z;q)=\prod_{n\geq1}\left(1-\frac{z^2}{\alpha^2_{\nu,n}(q)}\right),
	\end{equation*}
	where $\alpha_{\nu,n}(q)$ is the $n^\text{th}$ positive zero of $zJ'_\nu(z;q)+(1-\nu)J_\nu(z;q)$ satisfying $\alpha_{\nu,1}(q)<\alpha_{\nu,2}(q)<\cdots$ and hence we have the relation
	\begin{equation}\label{Bessel eqn dg2}
	\begin{split}
	1+\frac{zg''_\nu(z;q)}{g'_\nu(z;q)}&=1+\frac{z^2J''_\nu(z;q)+2(1-\nu)zJ'_\nu(z;q)+\nu(\nu-1)J_\nu(z;q)}{zJ'_\nu(z;q)+(1-\nu)J_\nu(z;q)}\\
	&=1-\sum_{n\geq1}\frac{2z^2}{\alpha^2_{\nu,n}(q)-z^2}.
	\end{split}
	\end{equation}
	We mimic the proof of Theorem \ref{lem starlike q Bessel} for $|z|<\alpha_{\nu,1}(q)$ so that  $g_\nu(z;q)$ satisfies the inequality \[ \left|\left(1+\frac{zg''_\nu(z;q)}{g'_\nu(z;q)}\right)^2 -1\right| \leq \left(\frac{|z|g''_\nu(|z|;q)}{g'_\nu(|z|;q)}\right)^2 -2\left(\frac{|z|g''_\nu(|z|;q)}{g'_\nu(|z|;q)}\right).   \] Hence, it follows that the lemniscate convex radius of $g_\nu(\cdot;q)$, $r^c_\mathcal{L}(g_\nu)$, is the unique positive root of the equation \[\left(\frac{rg''_\nu(r;q)}{g'_\nu(r;q)}\right)^2 -2\left(\frac{rg''_\nu(r;q)}{g'_\nu(r;q)}\right) -1=0\] in $(0,\alpha_{\nu,1}(q))$. The above equation is equivalent to equation \eqref{lem convex Bessel eqn gs} with $s=2$. Similarly, using the Hadamard factorization of $h'_\nu(z;q)$ given by $h'_\nu(z;q) =\prod_{n\geq1} (1-z/\beta^2_{\nu,n}(q))$, where $\beta_{\nu,n}(q)$ is the $n^\text{th}$ positive zero of $zJ'_\nu(z;q)+(2-\nu)J_\nu(z;q)$ such  that $\beta_{\nu,1}(q) <\beta_{\nu,2}(q) <\cdots$, the radius of lemniscate convexity for $h_\nu(\cdot;q)$, $r^c_\mathcal{L}(h_\nu)$, is the unique root of equation \eqref{lem convex Bessel eqn hs} for $s=2$ in $(0,\beta_{\nu,1}(q))$.
\end{proof}

%>>>>>>>>>>>>>>>>>>>>>>>>>>>>>>>>>>>>>>>>>>>>>>>>>>>>>>>>>>>>>>>>>>>>>>>>>>>>>>>>>>>>>>>>>>>>>>>>>>>>
It is known that the Bessel function of first kind admits the Weierstrass Decomposition \cite[p.498]{MR1349110} given by
\begin{equation*}\label{jv} J_\nu(z)=\frac{z^\nu}{2^\nu\Gamma(\nu+1)}\prod_{n\geq1}\left(1-\frac{z^2}{\xi_{\nu,n}^2}\right), \end{equation*}
where $\xi_{\nu,n}$ denotes the $n^\text{th}$ positive zero of the Bessel function $J_\nu$. Also, the zeros of the Bessel function of first kind satisfy the inequality $\xi_{\nu,1}<\xi_{\nu,2}<\cdots$ for $\nu>-1$. This infinite product is uniformly convergent on the compact subsets of $\mathbb{C}$. We have that Jackson and Hahn-Exton $q$-Bessel functions are $q$-extensions of classical Bessel function of first kind such that $J^{(2)}_\nu((1-z)q;q)\rightarrow J_\nu(z)$ and $J^{(3)}_\nu((1-z)q;q)\rightarrow J_\nu(2z)$ as $q\nearrow1$ for $\nu>-1$. Thus, the radii of lemniscate starlikeness of normalizations of Bessel function of first kind, given by \eqref{Bessel norm f}, \eqref{Bessel norm g} and \eqref{Bessel norm h}, are obtained as follows.
\begin{corollary}
	Let $\xi_{\nu,1}$ denote the first positive zero of the Bessel function of first kind $J_\nu$. If $\nu>0$, the function $f_\nu$ and if $\nu>-1$, the functions $g_\nu$ and $h_\nu$ have their radii of lemniscate starlikeness, $r^*_\mathcal{L}(f_\nu)$, $r^*_\mathcal{L}(g_\nu)$ and $r^*_\mathcal{L}(h_\nu)$, respectively to be unique positive root of the equations (respectively)
	\begin{align*}
	&r^2(J'_\nu(r))^2-4r\nu J'_\nu(r)J'_\nu(r)+2\nu^2(J_\nu(r))^2=0,\\
	&(rJ'_\nu(r)-\nu J_\nu(r))^2-2(r J'_\nu(r)-\nu J_\nu(r))J_\nu(r)-(J_\nu(r))^2=0
	\intertext{and}
	&(\sqrt{r}J'_\nu(\sqrt{r})-\nu J_\nu(\sqrt{r}))^2-4(\sqrt{r} J'_\nu(\sqrt{r})-\nu J_\nu(\sqrt{r})) J_\nu(\sqrt{r})-4(J_\nu(\sqrt{r}))^2=0
	\end{align*} in $(0,\xi_{\nu,1})$.
\end{corollary}
In the next result, the lemniscate convex radii of the normalizations of Bessel function of first kind, given by \eqref{Bessel norm f}, \eqref{Bessel norm g} and \eqref{Bessel norm h}, are determined.
\begin{corollary}
	Let $\nu>-1$. For the Bessel function of first kind $J_\nu$, if $\nu>0$, then the lemniscate convex radius of the function $f_\nu$; $r^c_\mathcal{L}(f_\nu)$, is the unique positive root of the equation \[\left(\dfrac{rJ''_\nu(r)}{J'_\nu(r)}+\left(\dfrac{1}{\nu}-1\right)\dfrac{rJ'_\nu(r)}{J_\nu(r)}\right)^2-2\left(\dfrac{rJ''_\nu(r)}{J'_\nu(r)}+\left(\dfrac{1}{\nu}-1\right)\dfrac{rJ'_\nu(r)}{J_\nu(r)}\right)-1=0\] in $(0,\xi'_{\nu,1})$, where $\xi'_{\nu,1}$ is the first positive zero of $J'_\nu$ and that for $g_\nu$; $r^c_\mathcal{L}(g_\nu)$ is the unique positive root of the equation
	\begin{equation*}
	\begin{split}	&\left(\frac{r^2J''_\nu(r)+2(1-\nu)rJ'_\nu(r)+\nu(\nu-1)J_\nu(r)}{rJ'_\nu(r)+(1-\nu)J_\nu(r)}\right)^2\\
	&\quad{}\quad{}-2\left(\frac{r^2J''_\nu(r)+2(1-\nu)rJ'_\nu(r)+\nu(\nu-1)J_\nu(r)}{rJ'_\nu(r)+(1-\nu)J_\nu(r)}\right)-1=0
	\end{split}
	\end{equation*} in $(0,\alpha_{\nu,1})$, where $\alpha_{\nu,1}$ is the first positive zero of $zJ'_\nu(z)+(1-\nu)J_\nu(z)$. For the function $h_\nu$, it is $r^c_\mathcal{L}(h_\nu)$, the unique positive root of the equation
	\begin{equation*}
	\begin{split} &\left(\frac{rJ''_\nu(\sqrt{r})+(3-2\nu)\sqrt{r}J'_\nu(\sqrt{r})+\nu(\nu-2)J_\nu(\sqrt{r})}{2(\sqrt{r}J'_\nu(\sqrt{r})+(2-\nu)J_\nu(\sqrt{r}))}\right)^2\\
	&\quad{}\quad{}-2\left(\frac{rJ''_\nu(\sqrt{r})+(3-2\nu)\sqrt{r}J'_\nu(\sqrt{r})+\nu(\nu-2)J_\nu(\sqrt{r})}{2(\sqrt{r}J'_\nu(\sqrt{r})+(2-\nu)J_\nu(\sqrt{r}))}\right)-1=0
	\end{split}
	\end{equation*} in $(0,\beta_{\nu,1})$, where $\beta_{\nu,1}$ is the first positive zero of $zJ'_\nu(z)+(2-\nu)J_\nu(z)$. Also, the inequalities $r^c_\mathcal{L}(f_\nu) <\xi'_{\nu,1} <\xi_{\nu,1}$, $r^c_\mathcal{L}(g_\nu) <\alpha_{\nu,1} <\xi_{\nu,1}$ and $r^c_\mathcal{L}(h_\nu) <\beta_{\nu,1} <\xi_{\nu,1}$ hold, where $\xi_{\nu,1}$ denote the first positive zero of $J_\nu$.
\end{corollary}

%>>>>>>>>>>>>>>>>>>>>>>>>>>>>>>>>>>>>>>>>>>>>>>>>>>>>>>>>>>>>>>>>>>>>>>>>>>>>>>>>>>>>>>>>>>>>>>>>>>
The radii of lemniscate starlikeness and lemniscate convexity for normalized Lommel functions are determined and the following lemma from \cite{MR3572758}, is used to prove the main results.
\begin{lemma}\label{lemma Lommel}
	Let \[ \varphi_k(z)={}_1F_2\left(1;\frac{\mu-k+2}{2},\frac{\mu+k+3}{2};\frac{-z^2}{4}\right),\] where $z\in\mathbb{C}$, $\mu\in\mathbb{R}$ and $k\in\{0,1,2,\ldots\}$ such that $\mu-k\not\in\{0,-1,\ldots\}$. Then, $\varphi_k$ is an entire function of order $\rho=1$. Consequently, the Hadamard's factorization of $\varphi_k$ is of the form \[ \varphi_k(z)=\prod_{n\geq1}\left(1-\frac{z^2}{z^2_{\mu,k,n}}\right),  \] where $z_{\mu,k,n}$ is the $n^\text{th}$ positive zero of the function $\varphi_k$, and the infinite product is absolutely convergent. Moreover, for $z$, $\mu$ and $k$ as above, we have $ (\mu-k+1)\varphi_{k+1}(z)=(\mu-k+1)\varphi_k(z)+z\varphi'_k(z)$. Also,
	\begin{equation*}\label{phi k}
	\sqrt{z}s_{\mu-k-\frac{1}{2},\frac{1}{2}}(z)=\frac{z^{\mu-k+1}}{(\mu-k)(\mu-k+1)}\varphi_k(z).
	\end{equation*}
\end{lemma}

The following result presents the radii of lemniscate starlikeness for the normalized Lommel functions $f_\mu$, $g_\mu$ and $h_\mu$ given by \eqref{Lommel norm f}, \eqref{Lommel norm g} and \eqref{Lommel norm h}, respectively.
\begin{theorem}\label{lem starlike Lommel}
	Let $\mu\in(-1,1)\setminus\{0\}$. Then the lemniscate starlike radius of the normalized Lommel  function of first kind $f_\mu$; $r^*_\mathcal{L}(f_\mu)$ is given by the smallest positive root of the equation
	\begin{equation}\label{lem star Lommel eqn fs}
	\begin{split}
	&\frac{1}{\left(\mu+\frac{1}{2}\right)^2}\left(\frac{rs'_{\mu-\frac{1}{2},\frac{1}{2}}(r)}{s_{\mu-\frac{1}{2},\frac{1}{2}}(r)}\right)^2-2=0\quad{}\text{for }-1<\mu<\frac{-1}{2}\\
	&\frac{1}{\left(\mu+\frac{1}{2}\right)^2}\left(\frac{rs'_{\mu-\frac{1}{2},\frac{1}{2}}(r)}{s_{\mu-\frac{1}{2},\frac{1}{2}}(r)}\right)^2-\frac{4}{\left(\mu+\frac{1}{2}\right)}\frac{rs'_{\mu-\frac{1}{2},\frac{1}{2}}(r)}{s_{\mu-\frac{1}{2},\frac{1}{2}}(r)}+2=0\quad{}\text{for}\quad{}\frac{-1}{2}<\mu<1.
	\end{split}
	\end{equation}
	For the functions $g_\mu$ and $h_\mu$; $r^*_\mathcal{L}(g_\mu)$ and $r^*_\mathcal{L}(h_\mu)$ are the smallest positive roots of the equations (respectively)
	\begin{align}
	&\left(\frac{rs'_{\mu-\frac{1}{2},\frac{1}{2}}(r)}{s_{\mu-\frac{1}{2},\frac{1}{2}}(r)}\right)^2-(2\mu+3)\left(\frac{rs'_{\mu-\frac{1}{2},\frac{1}{2}}(r)}{s_{\mu-\frac{1}{2},\frac{1}{2}}(r)}\right)+\mu^2+3\mu+\frac{1}{4}=0\label{lem star Lommel eqn gs}
	\intertext{and}
	&4\left(\frac{\sqrt{r}s'_{\mu-\frac{1}{2},\frac{1}{2}}(\sqrt{r})}{s_{\mu-\frac{1}{2},\frac{1}{2}}(\sqrt{r})}\right)^2-4(5+2\mu)\left(\frac{\sqrt{r}s'_{\mu-\frac{1}{2},\frac{1}{2}}(\sqrt{r})}{s_{\mu-\frac{1}{2},\frac{1}{2}}(\sqrt{r})}\right)+4\mu^2+20\mu-7=0. \label{lem star Lommel eqn hs}
	\end{align}
	Moreover, if $\tau$ represents the first positive zero of Lommel function of first kind $s_{\mu-\frac{1}{2},\frac{1}{2}}$, then the radii $r^*_\mathcal{L}(f_\mu)$, $r^*_\mathcal{L}(g_\mu)$ and $r^*_\mathcal{L}(h_\mu)$ are the unique roots of equations \eqref{lem star Lommel eqn fs}, \eqref{lem star Lommel eqn gs} and \eqref{lem star Lommel eqn hs}, respectively in $(0,\tau)$.
\end{theorem}
\begin{proof}
	We divide the proof for the function $f_\mu$ into the cases when $\mu\in(0,1)$ and when $\mu\in(-1,0)$. First suppose that $\mu\in(0,1)$. Since the zeros of the function $\varphi_0$ are real, let $\xi_{\mu,n}$ be the $n^\text{th}$ positive zero of the $\varphi_0$ such that $\xi_{\mu,1}<\xi_{\mu,2}<\cdots$. Thus, using Lemma \ref{lemma Lommel}, the Hadamard factorization of the Lommel function of first kind is given by
	\begin{equation*}\label{Lommel s}
	s_{\mu-\frac{1}{2},\frac{1}{2}}(z)=\frac{z^{\mu+\frac{1}{2}}}{\mu(\mu+1)}\prod_{n\geq 1}\left(1-\frac{z^2}{\xi_{\mu,n}^2}\right).
	\end{equation*} From the above equation, it follows that
	\begin{equation*}\label{Lommel expr s}
	\frac{zs'_{\mu-\frac{1}{2},\frac{1}{2}}(z)}{s_{\mu-\frac{1}{2},\frac{1}{2}}(z)}=\mu+\frac{1}{2}-\sum_{n\geq1}\frac{2z^2}{\xi_{\mu,n}^2-z^2}.
	\end{equation*}
	Since $f_\mu$ is a normalization of the function $s_{\mu-\frac{1}{2},\frac{1}{2}}$ given in \eqref{Lommel norm f}, we have the relation
	\begin{equation}\label{Lommel expr f}
	\frac{zf'_\mu(z)}{f_\mu(z)}=\frac{1}{\mu+\frac{1}{2}}\frac{zs'_{\mu-\frac{1}{2},\frac{1}{2}}(z)}{s_{\mu-\frac{1}{2},\frac{1}{2}}(z)}=1-\frac{1}{\mu+\frac{1}{2}}\sum_{n\geq1}\frac{2z^2}{\xi_{\mu,n}^2-z^2}.
	\end{equation}
	Following the procedure as in Theorem \ref{lem starlike q Bessel}, it is clear that the radius of lemniscate starlikeness for the function $f_\mu$; $r^*_\mathcal{L}(f_\mu)$, is the smallest positive root of the equation \[\left(\frac{rf'_\mu(r)}{f_\mu(r)}\right)^2 -4\left(\frac{rf'_\mu(r)}{f_\mu(r)}\right) +2=0. \]
	
	If $\mu\in(-1,0)$, then $\mu+1\in(0,1)$. Using the same technique as in proof of Theorem \ref{lem starlike q Bessel} for $\mu\in(0,1)$, substituting $\mu$ by $\mu-1$ and $\varphi_0$ by $\varphi_1$ with $\zeta_{\mu,n}$ being the $n^\text{th}$ positive zero of the function $\varphi_1$, we get
	\begin{align*}
	\left|\left(\frac{zf'_{\mu-1}(z)}{f_{\mu-1}(z)}\right)^2-1\right|&\leq \frac{1}{\left(\mu-\frac{1}{2}\right)^2}\left(\sum_{n\geq 1}\frac{2|z|^2}{\zeta^2_{\mu,n}-|z|^2}\right)^2+\frac{2}{\left|\mu-\frac{1}{2}\right|}\left(\sum_{n\geq 1}\frac{2|z|^2}{\zeta^2_{\mu,n}-|z|^2}\right)
	\end{align*} holds for $\mu\in(0,1)$ with $\mu\neq 1/2$. Simplifying the above equation, it is easy to see that \[  \left|\left(\frac{zf'_{\mu-1}(z)}{f_{\mu-1}(z)}\right)^2-1\right|\leq \begin{cases}
	&\frac{1}{\left(\mu-\frac{1}{2}\right)^2}\left(\sum_{n\geq 1}\frac{2|z|^2}{\zeta^2_{\mu,n}-|z|^2}\right)^2-\frac{2}{\left(\mu-\frac{1}{2}\right)}\left(\sum_{n\geq 1}\frac{2|z|^2}{\zeta^2_{\mu,n}-|z|^2}\right) \text{ for } \mu\in(0,\frac{1}{2}),\\
	&\frac{1}{\left(\mu-\frac{1}{2}\right)^2}\left(\sum_{n\geq 1}\frac{2|z|^2}{\zeta^2_{\mu,n}-|z|^2}\right)^2+\frac{2}{\left(\mu-\frac{1}{2}\right)}\left(\sum_{n\geq 1}\frac{2|z|^2}{\zeta^2_{\mu,n}-|z|^2}\right) \text{ for } \mu\in(\frac{1}{2},1).
	\end{cases} \]
	Using above inequality and equation \eqref{Lommel expr f}, replacing $z$ by $|z|$, the relation \[ \left|\left(\frac{zf'_{\mu-1}(z)}{f_{\mu-1}(z)}\right)^2-1\right|\leq  \begin{cases}
	&\left(\dfrac{|z|f'_{\mu-1}(|z|)}{f_{\mu-1}(|z|)}\right)^2-1 \quad\text{for}\quad \mu\in(0,\frac{1}{2}),\\
	&\left(\dfrac{|z|f'_{\mu-1}(|z|)}{f_{\mu-1}(|z|)}\right)^2-4\left(\dfrac{|z|f'_{\mu-1}(|z|)}{f_{\mu-1}(|z|)}\right)+3 \quad \text{for}\quad \mu\in(\frac{1}{2},1)
	\end{cases}  \]
	holds. Thus, on replacing $\mu$ by $\mu+1$ and combining both the cases, we observe that the function $f_\mu(z)$ is lemniscate starlike for $|z|<r^*$, where $r^*$ is the smallest root of the equation
	\begin{equation}\label{lem star Lommel eqn f}
	\begin{split}
	&\left(\frac{rf'_{\mu}(r)}{f_{\mu}(r)}\right)^2-2=0\quad{}\text{for}\quad{}-1<\mu<\frac{-1}{2}\\
	&\left(\frac{rf'_{\mu}(r)}{f_{\mu}(r)}\right)^2-4\left(\frac{rf'_{\mu}(r)}{f_{\mu}(r)}\right)+2=0 \quad{}\text{for}\quad{}\frac{-1}{2}<\mu<1\text{,}\quad{}\mu\neq 0.
	\end{split}
	\end{equation}
	
	Moreover, for $-1<\mu<1$, $\mu\neq0$ with $\xi_{\mu,1}$ being the first positive zero of $s_{\mu-\frac{1}{2},\frac{1}{2}}$, the function $u_\mu\colon(0,\xi_{\mu,1})\to\mathbb{R}$, defined by
	\[ u_\mu(r):=\begin{cases}
	&\left(\dfrac{rf'_{\mu}(r)}{f_{\mu}(r)}\right)^2-2=0\quad{}\text{for}\quad{}-1<\mu<\dfrac{-1}{2}\\
	&\left(\dfrac{rf'_{\mu}(r)}{f_{\mu}(r)}\right)^2-4\left(\dfrac{rf'_{\mu}(r)}{f_{\mu}(r)}\right)+2=0 \quad{}\text{for}\quad{}\dfrac{-1}{2}<\mu<1,\quad{}(\mu\neq 0)
	\end{cases}  \]
	is $u_\mu$ is continuous and $\lim_{r\searrow0}u_\mu(r)=-1<0$ and $\lim_{r\nearrow\xi_{\mu,1}}u_\mu(r)=\infty>0$. This ensures the existence of a zero of $u_\mu$ in $(0,\xi_{\mu,1})$. Furthermore, $u_\mu$ is strictly increasing as \[
	u'_\mu(r)=\begin{cases}
	\dfrac{-2}{\mu+\frac{1}{2}}\left(1-\dfrac{1}{\mu+\frac{1}{2}}\sum\limits_{n\geq1}\dfrac{2r^2}{\xi_{\mu,n}^2-r^2}\right)\left(\sum\limits_{n\geq1}\dfrac{4r\xi_{\mu,n}^2}{(\xi_{\mu,n}^2-r^2)^2}\right)\text{ for }-1<\mu<\dfrac{-1}{2}\\
	\dfrac{2}{\mu+\frac{1}{2}}\left(1+\dfrac{1}{\mu+\frac{1}{2}}\sum\limits_{n\geq1}\dfrac{2r^2}{\xi_{\mu,n}^2-r^2}\right)\left(\sum\limits_{n\geq1}\dfrac{4r\xi_{\mu,n}^2}{(\xi_{\mu,n}^2-r^2)^2}\right)\text{ for }\dfrac{-1}{2}<\mu<1,\quad (\mu\neq 0)
	\end{cases} \] is positive in $(0,\xi_{\mu,1})$. Thus, the lemniscate starlike radius of $f_\mu$ is the unique root of \eqref{lem star Lommel eqn f} in $(0,\xi_{\mu,1})$. It is clear using equation \eqref{Lommel expr f} that the zeros of equation \eqref{lem star Lommel eqn f} coincide with those of equation \eqref{lem star Lommel eqn fs}. Therefore, the radius of lemniscate starlikeness of the function $f_\mu$ is the unique positive zero of equation \eqref{lem star Lommel eqn fs} in $(0,\xi_{\mu,1})$.
	
	Now since $g_\mu$ and $h_\mu$ are normalizations of the Lommel function of first kind as mentioned in \eqref{Lommel norm g} and \eqref{Lommel norm h}, thus for $\mu\in(-1,1)\setminus\{0\}$, we have
	\begin{equation}\label{Lommel expr g}
	\frac{zg'_\mu(z)}{g_\mu(z)}=-\mu+\frac{1}{2}+\frac{zs'_{\mu-\frac{1}{2},\frac{1}{2}}(z)}{s_{\mu-\frac{1}{2},\frac{1}{2}}(z)}=1-\sum_{n\geq1}\frac{2z^2}{\xi_{\mu,n}^2-z^2}
	\end{equation} and
	\begin{equation}\label{Lommel expr h}
	\frac{zh'_\mu(z)}{h_\mu(z)}=\frac{3-2\mu}{4}+\frac{1}{2}\frac{\sqrt{z}s'_{\mu-\frac{1}{2},\frac{1}{2}}(\sqrt{z})}{s_{\mu-\frac{1}{2},\frac{1}{2}}(\sqrt{z})}=1-\sum_{n\geq 1}\frac{z}{\xi^2_{\mu,n}-z},
	\end{equation} where $\xi_{\mu,n}$ are the $n^\text{th}$ positive zero of the function $\varphi_0$ such that $\xi_{\mu,1}<\xi_{\mu,2}<\cdots$. Again as in Theorem \ref{lem starlike q Bessel}, it follows that the radii of lemniscate starlikeness of the functions $g_\mu(z)$ and $h_\mu(z)$ are the unique positive root of the equation \[ \left(\frac{rp'(r)}{p(r)}\right)^2-4\left(\frac{rp'(r)}{p(r)}\right)+2=0 \] in $(0,\xi_{\mu,1})$ and $(0,\sqrt{\xi_{\mu,1}})$, where the function $p$ collectively stands for the functions $g_\mu$ and $h_\mu$, respectively. Thus, the theorem holds.
\end{proof}

%>>>>>>>>>>>>>>>>>>>>>>>>>>>>>>>>>>>>>>>>>>>>>>>>>>>>>>>>>>>>>>>>>>>>>>>>>>>>>>>>>>>>>>>>>>>>>>>>>	
In the next theorem, we determine the radii of lemniscate convexity of the normalizations of Lommel function of first kind.
\begin{theorem}\label{lem convex Lommel}
	For the Lommel function of first kind $s_{\mu-\frac{1}{2},\frac{1}{2}}$, the lemniscate convex radius of the function $f_\mu$ with $\mu\in(-1/2,1)\setminus\{0\}$ and that of the functions $g_\mu$ and $h_\mu$ with $\mu\in(-1,1)\setminus\{0\}$; $r^c_\mathcal{L}(f_\mu)$, $r^c_\mathcal{L}(g_\mu)$ and $r^c_\mathcal{L}(h_\mu)$ are the smallest positive root of the equations (respectively)
	\begin{equation}\label{lem convex Lommel eqn fs}
	\begin{split}
	&\left(\frac{rs''_{\mu-\frac{1}{2},\frac{1}{2}}(r)}{s'_{\mu-\frac{1}{2},\frac{1}{2}}(r)}+\left(\frac{1}{\mu+\frac{1}{2}}-1\right)\frac{rs'_{\mu-\frac{1}{2},\frac{1}{2}}(r)}{s_{\mu-\frac{1}{2},\frac{1}{2}}(r)}\right)^2\\
	&\quad{}\quad{}\quad{}\quad{}-2\left(\frac{rs''_{\mu-\frac{1}{2},\frac{1}{2}}(r)}{s'_{\mu-\frac{1}{2},\frac{1}{2}}(r)}+\left(\frac{1}{\mu+\frac{1}{2}}-1\right)\frac{rs'_{\mu-\frac{1}{2},\frac{1}{2}}(r)}{s_{\mu-\frac{1}{2},\frac{1}{2}}(r)}\right)-1=0,
	\end{split}
	\end{equation}
	\begin{equation}\label{lem convex Lommel eqn gs}
	\begin{split}
	&\left(\frac{r^2s''_{\mu-\frac{1}{2},\frac{1}{2}}(r)+(1-2\mu)rs'_{\mu-\frac{1}{2},\frac{1}{2}}(r)+\left(\mu^2-\frac{1}{4}\right)s_{\mu-\frac{1}{2},\frac{1}{2}}(r)}{rs'_{\mu-\frac{1}{2},\frac{1}{2}}(r)+\left(-\mu+\frac{1}{2}\right)s_{\mu-\frac{1}{2},\frac{1}{2}}(r)}\right)^2  \\
	&\quad{}\quad{}\quad{}\quad{}-2\left(\frac{r^2s''_{\mu-\frac{1}{2},\frac{1}{2}}(r)+(1-2\mu)rs'_{\mu-\frac{1}{2},\frac{1}{2}}(r)+\left(\mu^2-\frac{1}{4}\right)s_{\mu-\frac{1}{2},\frac{1}{2}}(r)}{rs'_{\mu-\frac{1}{2},\frac{1}{2}}(r)+\left(-\mu+\frac{1}{2}\right)s_{\mu-\frac{1}{2},\frac{1}{2}}(r)}\right)-1=0
	\end{split}\end{equation}
	and
	\begin{equation}\label{lem convex Lommel eqn hs}
	\begin{split}
	&\left(\frac{rs''_{\mu-\frac{1}{2},\frac{1}{2}}(\sqrt{r})+2(1-\mu)\sqrt{r}s'_{\mu-\frac{1}{2},\frac{1}{2}}(\sqrt{r})+\frac{(2\mu-3)(2\mu+1)}{4}s_{\mu-\frac{1}{2},\frac{1}{2}}(\sqrt{r})}{2\left(\left(\frac{3}{2}-\mu\right)s_{\mu-\frac{1}{2},\frac{1}{2}}(\sqrt{r})+\sqrt{r}s'_{\mu-\frac{1}{2},\frac{1}{2}}(\sqrt{r})\right)}\right)^2\\
	&\quad{}\quad{}\quad{}-2\left(\frac{rs''_{\mu-\frac{1}{2},\frac{1}{2}}(\sqrt{r})+2(1-\mu)\sqrt{r}s'_{\mu-\frac{1}{2},\frac{1}{2}}(\sqrt{r})+\frac{(2\mu-3)(2\mu+1)}{4}s_{\mu-\frac{1}{2},\frac{1}{2}}(\sqrt{r})}{2\left(\left(\frac{3}{2}-\mu\right)s_{\mu-\frac{1}{2},\frac{1}{2}}(\sqrt{r})+\sqrt{r}s'_{\mu-\frac{1}{2},\frac{1}{2}}(\sqrt{r})\right)}\right)-1=0.
	\end{split}
	\end{equation}
	Furthermore, the radii $r^c_\mathcal{L}(f_\mu)$, $r^c_\mathcal{L}(g_\mu)$ and $r^c_\mathcal{L}(h_\mu)$ are unique roots of the equations \eqref{lem convex Lommel eqn fs}, \eqref{lem convex Lommel eqn gs} and \eqref{lem convex Lommel eqn hs} in $(0,\xi'_{\mu,1})$, $(0,\gamma_{\mu,1})$ and $(0,\delta_{\mu,1})$, respectively, where $\xi'_{\mu,1}$, $\xi_{\mu,1}$, $\gamma_{\mu,1}$ and $\delta_{\mu,1}$ are the first positive zeros of $s'_{\mu-\frac{1}{2},\frac{1}{2}}$, $s_{\mu-\frac{1}{2},\frac{1}{2}}$, $g'_\mu$ and $h'_\mu$, respectively.
\end{theorem}
\begin{proof}
	To show that the lemniscate convex radius of normalization $f_\mu$ is the smallest root of equation \eqref{lem convex Lommel eqn fs}, we divide the proof into cases when $0<\mu<1$ and $-1/2<\mu<0$.
	
	Firstly, assume that $\mu\in(0,1)$. It is known that for $0<\mu<1$, the zeros of $\varphi_0$, $\xi_{\mu,n}$, are real and positive  satisfying $\xi_{\mu,1}<\xi_{\mu,2}<\cdots$. Let $\xi'_{\mu,n}$ be the $n^\text{th}$ positive zero of $s'_{\mu-\frac{1}{2},\frac{1}{2}}$. The zeros of $s'_{\mu-\frac{1}{2},\frac{1}{2}}$ are interlaced with those of $s_{\mu-\frac{1}{2},\frac{1}{2}}$ as $0<\xi'_{\mu,1}<\xi_{\mu,1}<\xi'_{\mu,2}<\xi_{\mu,2}<\cdots$. On the similar lines as in proof of Theorem \ref{lem convex q Bessel}, with $\mu+1/2$ in place of $\nu$ and zeros $\xi_{\nu,n}(q)$ and $\xi'_{\nu,n}(q)$ of $J^{(2)}_\nu(\cdot;q)$ and $dJ^{(2)}_\nu(z;q)/dz$ replaced with $\xi_{\mu,n}$ and $\xi'_{\mu,n}$ of $s_{\mu-\frac{1}{2},\frac{1}{2}}$ and $s'_{\mu-\frac{1}{2},\frac{1}{2}}$, respectively (taking into account the cases $\mu\in(0,1/2]$ and $\mu\in[1/2,1)$ for which $\mu$ satisfies $1/(\mu+1/2)-1\geq0$ and $0\leq1-1/(\mu+1/2)<1/3$, respectively), it is observed that the radius of lemniscate convexity of $f_\mu$ is the unique root of  \eqref{lem convex Lommel eqn fs} in $(0,\xi'_{\mu,1})$. Now let $\mu\in(-1/2,0)$. Suppose that $\mu\in(1/2,1)$ and mimic the same proof with $\mu-1$ in place of $\mu$ and substituting $\varphi_0$ by $\varphi_1$ with $\zeta_{\mu,n}$ and $\zeta'_{\mu,n}$ being $n^\text{th}$ positive zeros of $s_{\mu-\frac{1}{2},\frac{1}{2}}$ and $s'_{\mu-\frac{1}{2},\frac{1}{2}}$, respectively. Then the function $f_{\mu-1}$ satisfies the inequality
	\begin{align*}
	\left|\left(1+\frac{zf''_{\mu-1}(z)}{f'_{\mu-1}(z)}\right)^2-1\right|
	&\leq\left(\frac{|z|f''_{\mu-1}(|z|)}{f'_{\mu-1}(|z|)}\right)^2-2\left(\frac{|z|f''_{\mu-1}(|z|)}{f'_{\mu-1}(|z|)}\right).
	\end{align*}
	On replacing $\mu$ with $\mu+1$, we get that the radius of lemniscate convexity of $f_\mu$, for $\mu\in(-1/2,1)\setminus\{0\}$ is the unique root of \eqref{lem convex Lommel eqn fs} in $(0,\xi'_{\mu,1})$, where $\xi'_{\mu,1}$ is the first positive zero of $s'_{\mu-\frac{1}{2},\frac{1}{2}}$.
	
	Since $g_\mu$ is a normalization of $s_{\mu-\frac{1}{2},\frac{1}{2}}(z)$ given by \eqref{Lommel norm g}, clearly $ g'_\mu(z) =\mu(\mu+1)z^{-\mu-\frac{1}{2}} ((-\mu+1/2) s_{\mu-\frac{1}{2},\frac{1}{2}}(z) +zs'_{\mu-\frac{1}{2},\frac{1}{2}}(z))$. For $-1<\mu<1$, with  $\gamma_{\mu,n}$ being the $n^\text{th}$ positive zero of $g'_\mu$ satisfying $\gamma_{\mu,1}<\gamma_{\mu,2}<\cdots$, the Hadamard factorization of $g'_\mu(z)$ is given by
	\begin{equation*}\label{Lommel dg}
	g'_\mu(z)=\prod_{n\geq1}\left(1-\frac{z^2}{\gamma^2_{\mu,n}}\right).
	\end{equation*}	Using the method as in proving Theorem \ref{lem convex q Bessel}, the radius of lemniscate convexity of the function $g_\mu$ is the unique root of the equation \eqref{lem convex Lommel eqn gs} in $(0,\gamma_{\mu,1})$. Similarly, it can be determined for the function $h_\mu$ given that the Hadamard factorization of $h'_\mu$ is $h'_\mu(z) =\prod_{n\geq1} (1-z/\delta^2_{\mu,n})$, where $\delta_{\mu,n}$ is the $n^\text{th}$ positive zero of $h'_\mu$ such that $\delta_{\mu,1}<\delta_{\mu,2}<\cdots$. Thus, the radius of lemniscate convexity of the function $h_\mu$ is the unique positive root of the equation \eqref{lem convex Lommel eqn hs} in $(0,\delta_{\mu,1})$. Hence, the theorem holds.
\end{proof}

%>>>>>>>>>>>>>>>>>>>>>>>>>>>>>>>>>>>>>>>>>>>>>>>>>>>>>>>>>>>>>>>>>>>>>>>>>>>>>>>>>>>>>>>>>>>>>>>>>
The Rodrigues formula of Legendre polynomial, $P_n(z)=(d^n((z^2-1)^n)/dz^n)/2^n n!$, implies that the Legendre polynomial of odd degree has only real roots and the roots of $\mathcal{P}_{2n-1}$ are $0$, $\pm \alpha_0,\pm \alpha_1,\pm \alpha_2,\cdots,\pm \alpha_{n-1}$ such that $0<\alpha_1<\alpha_2<\cdots<\alpha_{n-1}$. The product representation of $\mathcal{P}_{2n-1}$ as in \cite{MR3905364} is $ \mathcal{P}_{2n-1}(z)=a_{2n-1}z\prod_{k=1}^{n-1}(z^2-\alpha_k^2)$, Thus it is clear that
\begin{equation}\label{legendre eqn p}
\frac{z\mathcal{P}'_{2n-1}(z)}{\mathcal{P}_{2n-1}(z)}=1-\sum_{k=1}^{n-1}\frac{2z^2}{\alpha_k^2-z^2}.
\end{equation} Using the same technique as in  proof of Theorem \ref{lem starlike q Bessel}, we have the following.
\begin{theorem}
	Let $\alpha_1$ denote the first positive zero of the normalized Legendre polynomial of odd degree, $\mathcal{P}_{2n-1}$. Then the lemniscate starlike radius of $\mathcal{P}_{2n-1}$;  $r^*_\mathcal{L}(\mathcal{P}_{2n-1})$, is the unique positive root of the equation \[ \left(\frac{r\mathcal{P}'_{2n-1}(r)}{\mathcal{P}_{2n-1}(r)}\right)^2-4\left(\frac{r\mathcal{P}'_{2n-1}(r)}{\mathcal{P}_{2n-1}(r)}\right)+2=0  \]  in $(0,\alpha_1)$.
\end{theorem}

%>>>>>>>>>>>>>>>>>>>>>>>>>>>>>>>>>>>>>>>>>>>>>>>>>>>>>>>>>>>>>>>>>>>>>>>>>>>>>>>>>>>>>>>>>>>>>>>>>>
The lemniscate convex radius for the function $\mathcal{P}_{2n-1}$ is given by:
\begin{theorem}\label{lem convex legendre p}
	Let $\alpha_1$ be the first positive zero of the normalized Legendre polynomial of odd degree $\mathcal{P}_{2n-1}$.The lemniscate convex radius of the function $\mathcal{P}_{2n-1}$;  $r^c_\mathcal{L}(\mathcal{P}_{2n-1})$, is the smallest positive root of the equation
	\begin{equation}\label{lem convex legendre eqn p} \left(\frac{r\mathcal{P}''_{2n-1}(r)}{\mathcal{P}'_{2n-1}(r)}\right)^2-2\left(\frac{r\mathcal{P}''_{2n-1}(r)}{\mathcal{P}'_{2n-1}(r)}\right)-1=0
	\end{equation} in $(0,\alpha_1)$.
\end{theorem}
\begin{proof}
	For the normalized Legendre polynomial $\mathcal{P}_{2n-1}$, by means of equation \eqref{legendre eqn p}, it is clear to see that the equation
	\begin{equation}\label{legendre expr p}
	1+\frac{z\mathcal{P}''_{2n-1}(z)}{\mathcal{P}'_{2n-1}(z)}=1-\sum\limits_{k=1}^{n-1}\frac{2z^2}{z_k^2-z^2}-\frac{\sum_{k=1}^{n-1}\frac{4z^2z_k^2}{(z_k^2-z^2)^2}}{1-\sum_{k=1}^{n-1}\frac{2z^2}{z_k^2-z^2}}
	\end{equation} holds. For $|z|<\alpha_1$, using triangle's inequality in equation \eqref{legendre expr p}, we notice that the function $1+z\mathcal{P}''_{2n-1}(z)/\mathcal{P}'_{2n-1}(z)$ satisfies
	\begin{align*} \left|\left(1+\frac{z\mathcal{P}''_{2n-1}(z)}{\mathcal{P}'_{2n-1}(z)}\right)^2-1\right| &\leq \left(\sum_{k=1}^{n-1}\frac{2|z|^2}{z_k^2-|z|^2}+\frac{\sum_{k=1}^{n-1}\frac{4|z|^2z_k^2}{(z_k^2-|z|^2)^2}}{1-\sum_{k=1}^{n-1}\frac{2|z|^2}{z_k^2-|z|^2}} \right)^2\\ &\quad{}+2\left(\sum_{k=1}^{n-1}\frac{2|z|^2}{z_k^2-|z|^2}+\frac{\sum_{k=1}^{n-1}\frac{4|z|^2z_k^2}{(z_k^2-|z|^2)^2}}{1-\sum_{k=1}^{n-1}\frac{2|z|^2}{z_k^2-|z|^2}}\right).	\end{align*}	If $z$ is replaced with $|z|$ in equation \eqref{legendre expr p}, the above inequality results into
	\begin{equation}\label{lem convex legendre eqn mod p}
	\left|\left(1+\frac{z\mathcal{P}''_{2n-1}(z)}{\mathcal{P}'_{2n-1}(z)}\right)^2-1\right| \leq \left(\frac{|z|\mathcal{P}''_{2n-1}(|z|)}{\mathcal{P}'_{2n-1}(|z|)}\right)^2-2\left(\frac{|z|\mathcal{P}''_{2n-1}(|z|)}{\mathcal{P}'_{2n-1}(|z|)}\right).
	\end{equation}
	Let $r_1$ is the smallest positive root of the equation \eqref{lem convex legendre eqn p}, then the inequality \eqref{lem convex legendre eqn mod p} implies that for $|z|<r_1$, the Legendre polynomial $\mathcal{P}_{2n-1}$ satisfies $ |(1+z\mathcal{P}''_{2n-1}(z) /\mathcal{P}'_{2n-1}(z))^2-1|<1. $ Moreover, if the function $u_n\colon(0,\alpha_1)\to\mathbb{R}$ is defined by \[ u_n(r):=\left(\frac{r\mathcal{P}''_{2n-1}(r)}{\mathcal{P}'_{2n-1}(r)}\right)^2-2\left(\frac{r\mathcal{P}''_{2n-1}(r)}{\mathcal{P}'_{2n-1}(r)}\right)-1, \] then $u_n$ is continuous on $(0,\alpha_1)$ and $\lim_{r\searrow0}u_n(r)=-1<0$ and $\lim_{r\nearrow \alpha_1}u_n(r)=\infty>0$. Using the Intermediate Value Theorem, there exists a positive zero of $u_n$ in $(0,\alpha_1)$. Hence, the smallest positive root of $u_n(r)=0$ is the lemniscate convex radius of $\mathcal{P}_{2n-1}$, $r^c_\mathcal{L}(\mathcal{P}_{2n-1})$.
\end{proof}

\section{Janowski Starlikeness and Janowski Convexity}\label{janowski}
In this section, using the Hadamard factorizations of the specified functions, we determine the radii for which normalizations of these functions are Janowski starlike and Janowski convex.

The following result deals with the radii of Janowski starlikeness of the normalizations $f^{(s)}_\nu(z;q)$, $g^{(s)}_\nu(z;q)$ and $h^{(s)}_\nu(z;q)$ of the $q$-Bessel functions $J^{(s)}_\nu(z;q)$ for $s\in\{2,3\}$.
\begin{theorem}\label{jan star Bessel}
	For $s\in\{2,3\}$, let $\xi_{\nu,1}(q)$ and $\zeta_{\nu,1}(q)$ denote the first positive zeros of the $q$-Bessel functions $J^{(2)}_\nu(\cdot;q)$ and $J^{(3)}_\nu(\cdot;q)$, respectively where $0<q<1$. Suppose $\nu>-1$. Then the Janowski starlike radii of $f^{(s)}_\nu(\cdot;q)$ (with $\nu>0$), $g^{(s)}_\nu(\cdot;q)$ and $h^{(s)}_\nu(\cdot;q)$; $r^*_{A,B}(f^{(s)}_\nu)$, $r^*_{A,B}(g^{(s)}_\nu)$ and $r^*_{A,B}(h^{(s)}_\nu)$ are the smallest positive root of the equations (respectively)
	\begin{align}
	&\frac{rdJ^{(s)}_\nu(r;q)/dr}{J^{(s)}_\nu(r;q)}-\nu+\nu\left(\frac{A-B}{1+|B|}\right)=0,\label{jan star Bessel eqn fs}\\
	&\frac{rdJ^{(s)}_\nu(r;q)/dr}{J^{(s)}_\nu(r;q)}-\nu+\frac{A-B}{1+|B|}=0\label{jan star Bessel eqn gs}
	\intertext{and}
	&\frac{\sqrt{r}dJ^{(s)}_\nu(\sqrt{r};q)/dr}{J^{(s)}_\nu(\sqrt{r};q)}-\nu+2\frac{A-B}{1+|B|}=0.\label{jan star Bessel eqn hs}
	\end{align}
	Furthermore, for $s=2,3$, $r^*_{A,B}(f^{(s)}_\nu)$, $r^*_{A,B}(g^{(s)}_\nu)$ and $r^*_{A,B}(h^{(s)}_\nu)$ are unique positive roots of equations \eqref{jan star Bessel eqn fs}, \eqref{jan star Bessel eqn gs} and \eqref{jan star Bessel eqn hs}, respectively in $(0,t)$ such that $t=\xi_{\nu,1}(q)$ for $s=2$ and $t=\zeta_{\nu,1}(q)$ for $s=3$.
\end{theorem}
\begin{proof}
	The proof given here is for the case $s=2$ again and the other one follows on the similar lines. To determine the radius of Janowski starlikeness of the normalization $f_\nu(\cdot;q)$ of $J_\nu(\cdot;q)$ given by \eqref{Bessel norm f2}, we need to find a real positive number $r^*$ such that $|((zf'_\nu(z;q)/f_\nu(z;q))-1)/(A-Bzf'_\nu(z;q)/f_\nu(z;q))|<1$ for $|z|<r^*$. For $|z|<\xi_{\nu,1}(q)$, by equation \eqref{Bessel expr f2}, using triangle's inequality, it follows that the inequality \[ \left|\frac{\dfrac{zf'_\nu(z;q)}{f_\nu(z;q)}-1}{A-B\dfrac{zf'_\nu(z;q)}{f_\nu(z;q)}}\right| \leq \frac{\dfrac{1}{\nu} \sum\limits_{n\geq1} \dfrac{2|z|^2}{\xi^2_{\nu,n}(q)-|z|^2}}{A-B-|B|\dfrac{1}{\nu} \sum\limits_{n\geq1} \dfrac{2|z|^2}{\xi^2_{\nu,n}(q)-|z|^2}}  \] holds for $\nu>0$ with equality at $z=|z|=r$. By substituting $|z|$ in place of $z$ using equation \eqref{Bessel expr f2}, the above inequality yields
	\begin{equation}\label{jan star Bessel eqn mf2}
	\left|\frac{\dfrac{zf'_\nu(z;q)}{f_\nu(z;q)}-1}{A-B\dfrac{zf'_\nu(z;q)}{f_\nu(z;q)}}\right| \leq \frac{1-\dfrac{|z|f'_\nu(|z|;q)}{f_\nu(|z|;q)}}{A-B+|B|\left(\dfrac{|z|f'_\nu(|z|;q)}{f_\nu(|z|;q)}-1\right)}.
	\end{equation}
	If $r^*$ is the smallest positive root of the equation
	\begin{equation}\label{jan star Bessel eqn f2}
	\frac{rf'_\nu(r;q)}{f_\nu(r;q)}=1-\frac{A-B}{1+|B|},
	\end{equation} then the inequality \eqref{jan star Bessel eqn mf2} implies that the function $f_\nu$ is Janowski starlike for $|z|<r^*$.
	
	Further, the function $u_\nu(\cdot;q)\colon(0,\xi_{\nu,1}(q))\to\mathbb{R}$ defined by $ u_\nu(r;q):= rf'_\nu(r;q)/f_\nu(r;q)-1+(A-B)/(1+|B|)$ is continuous on $(0,\xi_{\nu,1}(q))$ and is strictly decreasing as \[ u'_\nu(r;q) =\dfrac{-1}{\nu} \sum_{n\geq1} \dfrac{4r\xi^2_{\nu,n}(q)}{(\xi^2_{\nu,n}(q)-r^2)^2} <0.\] Also, $\lim_{r\searrow0} u_\nu(r;q)=(A-B)/(1+|B|)>0$ and $\lim_{r\nearrow \xi_{\nu,1}(q)}u_\nu(r;q)=-\infty<0$. Therefore, the Intermediate Value Theorem ensures the existence of the unique root of $u_\nu(r;q)=0$ in $(0,\xi_{\nu,1}(q))$. Thus, the Janowski starlike radius of $f_\nu(\cdot;q)$; $r^*_{A,B}(f_\nu)$, is the unique root of equation \eqref{jan star Bessel eqn f2} in $(0,\xi_{\nu,1}(q))$ which is equivalent to \eqref{jan star Bessel eqn fs} with $s=2$.
	
	Similar equations hold for the normalizations $g_\nu$ and $h_\nu$ given by \eqref{Bessel norm g2} and \eqref{Bessel norm  h2}, respectively for $\nu>-1$. Hence, the radii of Janowski starlikeness for $g_\nu(\cdot;q)$ and $h_\nu(\cdot;q)$; $r^*_{A,B}(g_\nu)$ and $r^*_{A,B}(h_\nu)$, are the  unique roots of the equations \eqref{jan star Bessel eqn gs} and \eqref{jan star Bessel eqn hs}, respectively with $s=2$. Therefore, the theorem holds.
\end{proof}

%>>>>>>>>>>>>>>>>>>>>>>>>>>>>>>>>>>>>>>>>>>>>>>>>>>>>>>>>>>>>>>>>>>>>>>>>>>>>>>>>>>>>>>>>>>>>>>>>>>>
In the following result, we determine the radius of Janowski convexity of the normalized $q$-Bessel functions $f^{(s)}_\nu(\cdot;q)$, $g^{(s)}_\nu(\cdot;q)$ and $h^{(s)}_\nu(\cdot;q)$.
\begin{theorem}\label{jan convex Bessel}
	For the $q$-Bessel functions $J^{(2)}_\nu(\cdot;q)$ and $J^{(3)}_\nu(\cdot;q)$ with $0<q<1$, let $\xi'_{\nu,1}(q)$ and $\zeta'_{\nu,1}(q)$ be the first positive zeros of $dJ^{(s)}_\nu(z;q)/dz$ for $s=2,3$, respectively and $\alpha_{\nu,1}(q)$ and $\gamma_{\nu,1}(q)$ be the first positive zeros of $z\cdot dJ^{(s)}_\nu(z;q)/dz+(1-\nu)J^{(s)}_\nu(z;q)$ for $s=2$ and $s=3$, respectively. Similarly, let $\beta_{\nu,1}(q)$ and $\delta_{\nu,1}(q)$ denote the first positive zeros of $z\cdot dJ^{(s)}_\nu(z;q)/dz+(2-\nu)J^{(s)}_\nu(z;q)$ with $s=2$ and $s=3$, respectively. For $s\in\{2,3\}$, if $\nu>0$, then the function $f_\nu^{(s)}(\cdot;q)$ and if $\nu>-1$, then the functions $g^{(s)}_\nu(\cdot;q)$ and $h^{(s)}_\nu(\cdot;q)$, radii of Janowski convexity of these functions, $r^c_{A,B}(f^{(s)}_\nu)$, $r^c_{A,B}(g^{(s)}_\nu)$ and $r^c_{A,B}(h^{(s)}_\nu)$,  are the smallest positive root of the equations (respectively)
	\begin{align}
	&\frac{rd^2J^{(s)}_\nu(r;q)/dr^2}{dJ^{(s)}_\nu(r;q)/dr}+\left(\frac{1}{\nu}-1\right)\frac{rdJ^{(s)}_\nu(r;q)/dr}{J^{(s)}_\nu(r;q)}+\frac{A-B}{1+|B|}=0,\label{jan convex Bessel eqn fs}\\
	&\frac{\nu(\nu-1)J^{(s)}_\nu(r;q)+2(1-\nu)rdJ^{(s)}_\nu(r;q)/dr+r^2d^2J^{(s)}_\nu(r;q)/dr^2}{(1-\nu)J^{(s)}_\nu(r;q)+rdJ^{(s)}_\nu(r;q)/dr}+\frac{A-B}{1+|B|}=0\label{jan convex Bessel eqn gs}
	\end{align}
	and
	\begin{equation}\label{jan convex Bessel eqn hs}
	\frac{\nu(\nu-2)J^{(s)}_\nu(\sqrt{r};q)+(3-2\nu)\sqrt{r}dJ^{(s)}_\nu(\sqrt{r};q)/dr+rd^2J^{(s)}_\nu(\sqrt{r};q)/dr^2}{2((2-\nu)J^{(s)}_\nu(\sqrt{r};q)+\sqrt{r}dJ^{(s)}_\nu(\sqrt{r};q)/dr)}+\frac{A-B}{1+|B|}=0.
	\end{equation}
	In addition to that, for $s=2,3$, $r^c_{A,B}(f^{(s)}_\nu)$, $r^c_{A,B}(g^{(s)}_\nu)$ and $r^c_{A,B}(h^{(s)}_\nu)$ are the unique positive roots of equations \eqref{jan convex Bessel eqn fs}, \eqref{jan convex Bessel eqn gs} and \eqref{jan convex Bessel eqn hs} in $(0,t_1)$, $(0,t_2)$ and $(0,t_3)$, respectively such that $t_1=\xi'_{\nu,1}(q)$, $t_2=\alpha_{\nu,1}(q)$ and $t_3=\beta_{\nu,1}(q)$ for $s=2$ and $t_1=\zeta'_{\nu,1}(q)$, $t_2=\gamma_{\nu,1}(q)$ and $t_3=\delta_{\nu,1}(q)$ for $s=3$.
\end{theorem}
\begin{proof}
	The proof for the case $s=2$ is given as the other case is same except for the difference of zeros in the two cases. For $f_\nu(\cdot;q)$ to be Janowski convex in the disk $\{z: |z|<r\}$, the inequality $(zf''_\nu(z;q)/f'_\nu(z;q)) /(A-B(1+zf''_\nu(z;q)/f'_\nu(z;q)))|<1$ must hold for $|z|<r$. By basic computations, using \eqref{Bessel expr df2}, we notice that for $0\leq \nu <1$ the function $f_\nu(\cdot;q)$ satisfies the inequality
	\begin{align*}
	\left|\frac{\dfrac{zf''_\nu(z;q)}{f'_\nu(z;q)}}{A-B\left(1+\dfrac{zf''_\nu(z;q)}{f'_\nu(z;q)}\right)}\right|&\leq \frac{\sum\limits_{n\geq 1}\dfrac{2|z|^2}{\xi'^2_{\nu,n}(q)-|z|^2}+\left(\dfrac{1}{\nu}-1\right)\sum\limits_{n\geq 1}\dfrac{2|z|^2}{\xi^2_{\nu,n}(q)-|z|^2}}{A-B-|B|\left(\sum\limits_{n\geq 1}\dfrac{2|z|^2}{\xi'^2_{\nu,n}(q)-|z|^2}+\left(\dfrac{1}{\nu}-1\right)\sum\limits_{n\geq 1}\dfrac{2|z|^2}{\xi^2_{\nu,n}(q)-|z|^2}\right)}
	\end{align*} for $|z|<\xi'_{\nu,1}(q)$ with equality at $z=|z|=r$. By the means of the relation \eqref{relation}, the above inequality holds for $\nu\geq 1$ as well. From equation \eqref{Bessel expr df2}, replacing $z$ by $|z|$, the above inequality results into
	\begin{equation}\label{jan convex Bessel eqn mf2}
	\left|\frac{\dfrac{zf''_\nu(z;q)}{f'_\nu(z;q)}}{A-B\left(1+\dfrac{zf''_\nu(z;q)}{f'_\nu(z;q)}\right)}\right| \leq \frac{-\dfrac{|z|f''_\nu(|z|;q)}{f'_\nu(|z|;q)}}{A-B+|B|\dfrac{|z|f''_\nu(|z|;q)}{f'_\nu(|z|;q)}}
	\end{equation} for $\nu>0$. This implies that the radius of Janowski convexity for the function $f_\nu(z;q)$; $r^c_{A,B}(f_\nu)$ is the smallest positive root of the equation \begin{equation}\label{jan convex Bessel eqn f2} \frac{rf''_\nu(r;q)}{f'_\nu(r;q)}+\frac{A-B}{1+|B|}=0.
	\end{equation}
	
	On the other hand, for $\nu>0$, the function $u_\nu(\cdot;q)\colon(0,\xi'_{\nu,1}(q))\to\mathbb{R}$, defined by $u_\nu(r;q) :=rf''_\nu(r;q)/f'_\nu(r;q) +(A-B)/(1+|B|)$, is continuous on $(0,\xi'_{\nu,1}(q))$ and is strictly decreasing in $(0,\xi'_{\nu,1}(q))$ as \[ u'_\nu(r;q) <4r\sum_{n\geq1} \left(\frac{\xi^2_{\nu,n}(q)}{(\xi^2_{\nu,n}(q)-r^2)^2} -\frac{\xi'^2_{\nu,n}(q)}{(\xi'^2_{\nu,n}(q)-r^2)^2}\right) <0  \] since $\xi^2_{\nu,n}(q) (\xi'^2_{\nu,n}(q)-r^2)^2 <\xi'^2_{\nu,n}(q)(\xi^2_{\nu,n}(q)-r^2)^2$ for $r<\sqrt{\xi'_{\nu,1}(q)\xi_{\nu,1}(q)}$ and $\nu>0$. Also, $\lim_{r\searrow 0}u_\nu(r;q)=(A-B)/(1+|B|)>0$ and $\lim_{r\nearrow \xi'_{\nu,1}(q)} u_\nu(r;q)=-\infty<0$. Therefore, by monotonicity of the function $u_\nu(\cdot;q)$, it is clear that the function $f_\nu$ is Janowski convex for $|z|<r_1$ where $r_1$ is the unique positive root of equation \eqref{jan convex Bessel eqn f2}. Since the zeros of the equation \eqref{jan convex Bessel eqn f2} are same as those equation \eqref{jan convex Bessel eqn fs} with $s=2$, the radius of $f_\nu(\cdot;q)$, $r^c_{A,B}(f_\nu)$, is the unique positive root of equation \eqref{jan convex Bessel eqn fs} for $s=2$ in $(0,\xi'_{\nu,1}(q))$.
	
	For the function $g_\nu(\cdot;q)$ and $h_\nu(\cdot;q)$, keeping in view equation \eqref{Bessel eqn dg2} and proceeding according to the proof of Theorem \ref{jan star Bessel}, we get that the Janowski convex radii of $g_\nu(\cdot;q)$ and $h_\nu(\cdot;q)$, $r^c_{A,B}(g_\nu)$ and $r^c_{A,B}(h_\nu)$, are the unique positive roots of the equations \eqref{jan convex Bessel eqn gs} and \eqref{jan convex Bessel eqn hs} with $s=2$ in $(0,\alpha_{\nu,1}(q))$ and $(0,\beta_{\nu,1}(q))$, respectively.
\end{proof}

%>>>>>>>>>>>>>>>>>>>>>>>>>>>>>>>>>>>>>>>>>>>>>>>>>>>>>>>>>>>>>>>>>>>>>>>>>>>>>>>>>>>>>>>>>>>>>>>>>
Since the $q$-Bessel functions are analytic extensions of the Bessel function of first kind as stated before, therefore for the normalized Bessel functions $f_\nu$, $g_\nu$ and $h_\nu$ given by \eqref{Bessel norm f}, \eqref{Bessel norm g} and \eqref{Bessel norm h}, respectively, we have the following.
\begin{corollary}
	Let $\xi_{\nu,1}$ is the first positive zero of the function the Bessel function of first kind $J_\nu$. Suppose $\nu>-1$. Then the Janowski starlike radii of the function $f_\nu$ (with $\nu>0$), $g_\nu$ and $h_\nu$; $r^*_{A,B}(f_\nu)$, $r^*_{A,B}(g_\nu)$ and $r^*_{A,B}(h_\nu)$ are respectively the unique root of the equations (respectively) \[ \frac{rJ'_\nu(r)}{J_\nu(r)} -\nu +\nu \left(\frac{A-B}{1+|B|}\right) =0, \  \frac{rJ'_\nu(r)}{J_\nu(r)}-\nu+\frac{A-B}{1+|B|}=0   \text{ and } \frac{\sqrt{r}J'_\nu(\sqrt{r})}{J_\nu(\sqrt{r})}-\nu+2\frac{A-B}{1+|B|}=0. \]
\end{corollary}
\begin{corollary}
	For the Bessel function of first kind $J_\nu$, let $\xi'_{\nu,1}$, $\alpha_{\nu,1}$ and $\beta_{\nu,1}$ are the first positive zeros of $J'_\nu(z)$, $z J'_\nu(z)+(1-\nu)J_\nu(z)$ and $z J'_\nu(z)+(2-\nu)J_\nu(z)$, respectively. If $\nu>0$, the Janowski convex radius of the function $f_\nu$; $r^c_{A,B}(f_\nu)$, if $\nu>-1$, that of the functions $g_\nu$; $r^c_{A,B}(g_\nu)$ and $h_\nu$; $r^c_{A,B}(h_\nu)$ are the smallest roots of the equations (respectively)
	\begin{align*}
	&\frac{rJ''_\nu(r)}{J'_\nu(r)}+\left(\frac{1}{\nu}-1\right)\frac{rJ'_\nu(r)}{J_\nu(r)}+\frac{A-B}{1+|B|}=0,\\
	&\frac{\nu(\nu-1)J_\nu(r)+2(1-\nu)rJ'_\nu(r)+r^2J''_\nu(r)}{(1-\nu)J_\nu(r)+rJ'_\nu(r)}+\frac{A-B}{1+|B|}=0
	\intertext{and}
	&\frac{\nu(\nu-2)J_\nu(\sqrt{r})+(3-2\nu)\sqrt{r}J'_\nu(\sqrt{r})+rJ''_\nu(\sqrt{r})}{2((2-\nu)J_\nu(\sqrt{r})+\sqrt{r}J'_\nu(\sqrt{r}))}+\frac{A-B}{1+|B|}=0.
	\end{align*}
	Moreover, the radii $r^c_{A,B}(f_\nu)$, $r^c_{A,B}(g_\nu)$ and $r^c_{A,B}(h_\nu)$ are the  unique roots of the respective above equations in $(0,\xi'_{\nu,1})$, $(0,\alpha_{\nu,1})$ and $(0,\beta_{\nu,1})$.
\end{corollary}

%>>>>>>>>>>>>>>>>>>>>>>>>>>>>>>>>>>>>>>>>>>>>>>>>>>>>>>>>>>>>>>>>>>>>>>>>>>>>>>>>>>>>>>>>>>>>>>>>>>>
The next theorem presents the radius of Janowski starlikeness of the normalized Lommel functions of first kind given by \eqref{Lommel norm f}, \eqref{Lommel norm g} and \eqref{Lommel norm h}.
\begin{theorem}
	Let $\xi_{\mu,1}$ be the first positive zero of the Lommel function of first kind $s_{\mu-\frac{1}{2},\frac{1}{2}}$. Suppose $\mu\in(-1,1)\setminus\{0\}$. Then the Janowski starlike radii of the functions $f_\mu$, (with $-1/2<\mu<1,\mu\neq0$), $g_\mu$ and $h_\mu$, $r^*_{A,B}(f_\mu)$;  $r^*_{A,B}(g_\mu)$ and  $r^*_{A,B}(h_\mu)$, are the unique positive roots of the equations
	\begin{align}
	&\frac{rs'_{\mu-\frac{1}{2},\frac{1}{2}}(r)}{s_{\mu-\frac{1}{2},\frac{1}{2}}(r)}-\left(\mu+\frac{1}{2}\right)+\left(\mu+\frac{1}{2}\right)\frac{A-B}{1+|B|}=0,\label{jan star Lommel eqn fs}\\
	&\frac{rs'_{\mu-\frac{1}{2},\frac{1}{2}}(r)}{s_{\mu-\frac{1}{2},\frac{1}{2}}(r)}-\left(\mu+\frac{1}{2}\right)+\frac{A-B}{1+|B|}=0\label{jan star Lommel eqn gs}
	\intertext{and}
	&\frac{\sqrt{r}s'_{\mu-\frac{1}{2},\frac{1}{2}}(\sqrt{r})}{s_{\mu-\frac{1}{2},\frac{1}{2}}(\sqrt{r})}-\left(\mu+\frac{1}{2}\right)+2\frac{A-B}{1+|B|}=0\label{jan star Lommel eqn hs}
	\end{align}
	respectively in $(0,\xi_{\mu,1})$.
\end{theorem}
\begin{proof}
	The proof for the function $f_\mu$ is divided into two cases, one when $\mu\in(0,1)$ and the other when $\mu\in(-1/2,0)$.
	
	Suppose that $\mu\in(0,1)$. Then with the factorization as in equation \eqref{Lommel expr f}, if we mimic the proof of Theorem \ref{jan star Bessel}, the Janowski starlike radius of $f_\mu$ is obtained to be the smallest positive root of equation
	\begin{equation}\label{jan star Lommel eqn f}\frac{rf'_\mu(r)}{f_\mu(r)}-1+\frac{A-B}{1+|B|} =0. \end{equation}
	
	For the case when $\mu\in(-1/2,0)$, we repeat the above proof assuming $1/2<\mu<1$ and substituting $\mu$ with $\mu-1$ to get the inequality \[ \left|\frac{\dfrac{zf'_{\mu-1}(z)}{f_{\mu-1}(z)}-1}{A-B\dfrac{zf'_{\mu-1}(z)}{f_{\mu-1}(z)}}\right| \leq \frac{1-\dfrac{|z|f'_{\mu-1}(|z|)}{f_{\mu-1}(|z|)}}{A-B+|B|\left(\dfrac{|z|f'_{\mu-1}(|z|)}{f_{\mu-1}(|z|)}-1\right)}  \] for $|z|<\xi_{\mu,1}$. By substituting $\mu+1$ in place of $\mu$ in the above inequality, it is observed that the  function $f_\mu$ is Janowski starlike in $\{z:|z|<r^*\}$ where $r^*$ is unique root of equation \eqref{jan star Lommel eqn f} for $\mu\in (-1/2,0)$. On combining both the cases, it follows that for $\mu\in(-1/2,1)$, $f_\mu$ is Janowski starlike for $|z|<r^*_{A,B}(f_\mu)$, where $r^*_{A,B}(f_\mu)$ is the unique positive root of equation \eqref{jan star Lommel eqn f} in $(0,\xi_{\mu,1})$. An equivalent form of the above equation is given by equation \eqref{jan star Lommel eqn fs}.
	
	By virtue of equation \eqref{Lommel expr g}, for $\mu\in(-1,1)\setminus\{0\}$ with the proof similar to that of Theorem \ref{jan star Bessel}, clearly the radii of Janowski starlikeness of the functions $g_\mu$ and $h_\mu$ are the unique positive roots of the equations \eqref{jan star Lommel eqn gs} and \eqref{jan star Lommel eqn hs}, respectively in $(0,\xi_{\mu,1})$. Thus, the theorem holds.
\end{proof}

%>>>>>>>>>>>>>>>>>>>>>>>>>>>>>>>>>>>>>>>>>>>>>>>>>>>>>>>>>>>>>>>>>>>>>>>>>>>>>>>>>>>>>>>>>>>>>>>>>>
Now for the Janowski convexity of the normalizations of Lommel function of first kind, following the same steps as done in proof of the Theorem \ref{jan convex Bessel} with the similar notations as mentioned before, we get the following:
\begin{theorem}
	Suppose $\mu\in(-1,1)\setminus\{0\}$. For Lommel function of first kind $s_{\mu-\frac{1}{2},\frac{1}{2}}$, let $\xi'_{\mu,1}$, $\xi_{\mu,1}$, $\gamma_{\mu,1}$ and $\delta_{\mu,1}$ denote the first positive zeros of $s'_{\mu-\frac{1}{2},\frac{1}{2}}$, $s_{\mu-\frac{1}{2},\frac{1}{2}}$, $g'_\mu$ and $h'_\mu$, respectively. If $\mu\in(-1/2,1)\setminus\{0\}$, then the function $f_\mu$ and if $\mu\in(-1,1)\setminus\{0\}$, the functions $g_\mu$ and $h_\mu$, their radii of Janowski convexity $r^c_{A,B}(f_\mu)$, $r^c_{A,B}(g_\mu)$ and $r^c_{A,B}(h_\mu)$, respectively are the unique positive roots of the equations (respectively)	\[ \frac{rs''_{\mu-\frac{1}{2},\frac{1}{2}}(r)}{s'_{\mu-\frac{1}{2},\frac{1}{2}}(r)} -\left(\frac{\mu-\frac{1}{2}}{\mu+\frac{1}{2}}\right) \frac{rs'_{\mu-\frac{1}{2},\frac{1}{2}}(r)}{s_{\mu-\frac{1}{2},\frac{1}{2}}(r)} +\frac{A-B}{1+|B|}=0, \] \[ \frac{(3/2-\mu)rs'_{\mu-\frac{1}{2},\frac{1}{2}}(r)+r^2s''_{\mu-\frac{1}{2},\frac{1}{2}}(r)}{(1/2-\mu)s_{\mu-\frac{1}{2},\frac{1}{2}}(r)+rs'_{\mu-\frac{1}{2},\frac{1}{2}}(r)} -\left(\mu+\frac{1}{2}\right) +\frac{A-B}{1+|B|}=0  \] and \[ \frac{(5/2-\mu)\sqrt{r}s'_{\mu-\frac{1}{2},\frac{1}{2}}(\sqrt{r})+rs''_{\mu-\frac{1}{2},\frac{1}{2}}(\sqrt{r})}{(3/2-\mu)s_{\mu-\frac{1}{2},\frac{1}{2}}(\sqrt{r})+\sqrt{r}s'_{\mu-\frac{1}{2},\frac{1}{2}}(\sqrt{r})}-\left(\mu+\frac{1}{2}\right) +2\frac{A-B}{1+|B|}=0 \] in $(0,\xi'_{\mu,1})$, $(0,\gamma_{\mu,1})$ and $(0,\delta_{\mu,1})$, respectively. Moreover, the relations  $r^c_{A,B}(f_\mu) <\xi'_{\mu,1} <\xi_{\mu,1}$, $r^c_{A,B}(g_\mu) <\gamma_{\mu,1} <\xi_{\mu,1}$ and $r^c_{A,B}(h_\mu) <\delta_{\mu,1} <\xi_{\mu,1}$ hold.
\end{theorem}

For the normalized Legendre polynomial of odd degree $\mathcal{P}_{2n-1}$, by means of equation \eqref{legendre eqn p} and the proof of Theorem \ref{jan star Bessel}, we have that the inequality \[ \left|\frac{\dfrac{z\mathcal{P}'_{2n-1}(z)}{\mathcal{P}_{2n-1}(z)}-1}{A-B\dfrac{z\mathcal{P}'_{2n-1}(z)}{\mathcal{P}_{2n-1}(z)}}\right| \leq \frac{1-\dfrac{|z|\mathcal{P}'_{2n-1}(|z|)}{\mathcal{P}_{2n-1}(|z|)}}{A-B-|B|\left(1-\dfrac{|z|\mathcal{P}'_{2n-1}(|z|)}{\mathcal{P}_{2n-1}(|z|)}\right)} \] holds for $|z|<\alpha_1$, where $\alpha_1$ is the first positive zero of the polynomial $\mathcal{P}_{2n-1}$.

Therefore, the following theorem gives the radius of Janowski starlikeness of the function $\mathcal{P}_{2n-1}$.
\begin{theorem}
	Let $\alpha_1$ denote the first positive zero of the normalized Legendre polynomial of odd degree $\mathcal{P}_{2n-1}$. The Janowski starlike radius for the function $\mathcal{P}_{2n-1}(z)$; $r^*_{A,B}(\mathcal{P}_{2n-1})$, is the unique positive root of the equation \[ \frac{r\mathcal{P}'_{2n-1}(r)}{\mathcal{P}_{2n-1}(r)}-1+\frac{A-B}{1+|B|}=0\] in $(0,\alpha_1)$.
\end{theorem}
With the calculations as performed in proof of Theorem \ref{lem convex legendre p}, for the function $\mathcal{P}_{2n-1}$, if $|z|<\alpha_1$, then \[ \left|\frac{\dfrac{z\mathcal{P}''_{2n-1}(z)}{\mathcal{P}'_{2n-1}(z)}}{A-B\left(1+\dfrac{z\mathcal{P}''_{2n-1}(z)}{\mathcal{P}'_{2n-1}(z)}\right)}\right|\leq\frac{-\dfrac{|z|\mathcal{P}''_{2n-1}(|z|)}{\mathcal{P}'_{2n-1}(|z|)}}{A-B+|B|\dfrac{|z|\mathcal{P}''_{2n-1}(|z|)}{\mathcal{P}'_{2n-1}(|z|)}}.   \] Hence, the Janowski convex radius of the function $\mathcal{P}_{2n-1}$ is given by the following result.
\begin{theorem}
	Let $\alpha_1$ be the first positive zero of the normalized Legendre polynomial of odd degree $\mathcal{P}_{2n-1}$. The radius of Janowski convexity for the function $\mathcal{P}_{2n-1}(z)$; $r^c_{A,B}(\mathcal{P}_{2n-1})$, is the smallest positive root of the equation \[ \frac{r\mathcal{P}''_{2n-1}(r)}{\mathcal{P}'_{2n-1}(r)}+\frac{A-B}{1+|B|}=0 \] in $(0,\alpha_1)$.
\end{theorem}

%>>>>>>>>>>>>>>>>>>>>>>>>>>>>>>>>>>>>>>>>>>>>>>>>>>>>>>>>>>>>>>>>>>>>>>>>>>>>>>>>>>>>>>>>>>>>>>>>>

\end{document}